\documentclass[a4paper,11pt,reqno]{amsart}

\usepackage[headings]{fullpage}

\usepackage{mathtools}
\usepackage{dsfont,amssymb}
\usepackage{tensor}
\usepackage{tikz}
\usetikzlibrary{cd}
\usepackage[all]{xy}
\usepackage{xcolor}
\usepackage{epstopdf}
\usepackage{caption}
\usepackage{subcaption}

\usepackage{hyperref}

\usepackage[nobysame,alphabetic,initials,msc-links]{amsrefs}

\DefineSimpleKey{bib}{how}

\renewcommand{\eprint}[1]{#1}
\BibSpec{misc}{%
  +{}{\PrintAuthors}  {author}
  +{,}{ \textit}      {title}
  +{,}{ }             {how}
  +{}{ \parenthesize} {date}
  +{,} { available at \eprint}        {eprint}
  +{,}{ available at \url}{url}
  +{,}{ }             {note}
  +{.}{}              {transition}
}

\mathchardef\mhyph="2D

\numberwithin{equation}{section}

\newtheorem{theorem}{Theorem}[section]
\newtheorem{corollary}[theorem]{Corollary}
\newtheorem{lemma}[theorem]{Lemma}
\newtheorem{proposition}[theorem]{Proposition}
\theoremstyle{remark}
\newtheorem{remark}[theorem]{Remark}
\newtheorem{example}[theorem]{Example}
\theoremstyle{definition}
\newtheorem{definition}[theorem]{Definition}
\newtheorem{theoremAlph}{Theorem}

\newcommand\bp{\begin{proof}}
\newcommand\ep{\end{proof}}

\def\hyphen{{\hbox{-}}}
\def\-{{\hbox{-}}}

\newcommand{\C}{\mathbb C}
\newcommand{\N}{\mathbb N}
\newcommand{\R}{\mathbb R}
\newcommand{\Z}{\mathbb Z}
\newcommand{\bE}{\mathbb E}

\newcommand{\cA}{\mathcal A}
\newcommand{\cB}{\mathcal B}
\newcommand{\bB}{\mathbb B}
\newcommand{\bG}{\mathbb G}
\newcommand{\cC}{\mathcal C}
\newcommand{\cD}{\mathcal D}

\newcommand{\cF}{\mathcal F}

\newcommand{\cH}{\mathcal H}
\newcommand{\cK}{\mathcal K}

\newcommand{\cL}{\mathcal L}
\newcommand{\cM}{\mathcal M}
\newcommand{\cN}{\mathcal N}
\newcommand{\cO}{\mathcal O}

\newcommand{\cS}{\mathcal S}

\newcommand{\cZ}{\mathcal Z}
\newcommand{\into}{\hookrightarrow}
\newcommand{\colim}{\text{colim}}

\DeclareMathOperator{\Ad}{Ad}
\DeclareMathOperator{\Alg}{Alg}

\DeclareMathOperator{\End}{End}

\DeclareMathOperator{\Hom}{Hom}

\DeclareMathOperator{\Irr}{Irr}

\DeclareMathOperator{\Mod}{Mod}

\DeclareMathOperator{\Rep}{Rep}

\DeclareMathOperator{\Vect}{Vec}
\DeclareMathOperator{\Cat}{Cat}
\DeclareMathOperator{\Bal}{Bal}

\DeclareMathOperator{\tr}{tr}
\DeclareMathOperator{\Clin}{C*lin}
\DeclareMathOperator{\CAlg}{C*Alg}

\DeclareMathOperator{\Obj}{Obj}
\DeclareMathOperator{\Fun}{Fun}

\DeclareMathOperator{\Bim}{Bim}
\DeclareMathOperator{\II}{II}
\DeclareMathOperator{\III}{III}

\DeclareMathOperator{\Emb}{Emb}
\DeclareMathOperator{\Mfld}{Mfld}
\DeclareMathOperator{\Disk}{Disk}
\DeclareMathOperator{\Or}{or}

\DeclareMathOperator{\Rex}{Rex}

\DeclareMathOperator{\DisInc}{DisInc}

\DeclareMathOperator{\fd}{fd}

\newcommand\op{\text{op}}
\newcommand\un{{\mathds 1}}

\newcommand{\id}{\mathrm{id}}

\newcommand{\Hilb}{\mathrm{Hilb}}

\begin{document}

\title{C*-algebraic factorization homology and Realization of Cyclic Representations}

\date{}

\author{Lucas Hataishi}
\address{Universitetet i Oslo}
\email{lucasyh@math.uio.no}

\thanks{Supported by the NFR project 300837 ``Quantum Symmetry''.}

\begin{abstract}
We prove cocontinuity of the $\max$-tensor product of C*-categories and develop a framework to perform factorization homology in a C*-setting. In such context, we specialize some results of D. Ben-Zvi, A. Brochier and D. Jordan. As a consequence of our constructions, we realize quantum Hamiltonian reduction in terms of bimodules over a factor $N$. We also provide a GNS-type reconstruction theorem for C*-algebra objects  in categories of bimodules over a $\II_1$-factor, enhancing a realization theorem due to C. Jones and D. Penneys.
\end{abstract}

\maketitle

\tableofcontents

\section{Introduction}
\label{sec:intro}

This paper is an attempt of bringing ideas and techniques from geometric representation theory and topological quantum field theory (TQFT) to operator algebras. The guiding principle is that many interesting algebraic structures arise as linear representations of topological categories. An unital algebra, for instance, determines and is completely determined by a functor from the tensor category of oriented intervals to the category of vector spaces. Frobenius algebras are in bijection with tensor functors from the category of 2-dimensional cobordisms to the category of vector spaces. Homology theories can also be placed into this picture by reading off the Eilenberg-Steenrod axioms. It says that a homology theory is a tensor functor from the category of CW-complexes to the category of chain complexes, satisfying the excision formula. Every homology theory is determined by its value on a point, the simplest of all CW-complexes. In particular, every chain complex defines a homology theory for which its value on the point is the given chain complex.

Our goal in this work is to, following \cite{LurieClassification, AyalaFrancis,BenZviBrochierJordan1,BenZviBrochierJordan2}, produce homology like theories for surfaces. We shall consider the tensor category $\Mfld_2$ whose objects are compact oriented surfaces, morphisms are smooth orientation preserving embeddings, and the tensor structure is given by disjoint union. A homology theory for surfaces will then be a tensor functor $F: \Mfld_2 \to \cS$ taking values in a symmetric tensor category $\cS$, satisfying an excision formula. In $\Mfld_2$, the simplest object is the disk $D$. It is proven in \cite{AyalaFrancis} that, if $\cS$ is sifted cocomplete and the tensor structure is sifted cocontinuous, a tensor functor $F:\Mfld_2 \to \cS$ is determined by $F(D)$. The difference, comparing with classical homology theories, is that $F(D)$ is not a merely object in $\cS$, but comes with the structure of an $\bE_2$-algebra in $\cS$: operations parameterized by the spaces of embeddings $D^{\sqcup k} \into D$, where $\sqcup$ denotes disjoint union. Ayala and Francis prove in \cite{AyalaFrancis} the converse, that every $\bE_2$-algebra in $\cS$ determines a homology theory $F: \Mfld_2 \to \cS$.
The procedure that takes an $\bE_2$-algebra and builds the tensor functor $F$ is called {\em factorization homology}. It was first sketched in \cite{LurieClassification}, but the rigorous formulation is due to Ayala and Francis, and it works in any dimension. We shall however consider just the 2-dimensional case. The value $F(\Sigma)$ of factorization homology can be computed as a colimit in terms of the $\bE_2$-algebra $\cC:= F(D)$. For this reason, it is usually written
$$F(\Sigma) = \int_{\Sigma} \cC \ , $$
a convention that we will follow in this paper.

In \cite{BenZviBrochierJordan1,BenZviBrochierJordan2} it was considered the case when $\cS = \Rex$, the category of cocomplete linear categories with right exact functors as morphisms. The authors prove that it is cocomplete and that the Deligne-Kelly tensor product is cocontinuos. Cocompleteness and cocontinuity imply, respectively, sifted cocompleteness and sifted cocontinuity. They also establish several technical tools allowing for concrete computations. In $\Rex$, $\bE_2$-algebras are braided tensor categories. The main examples in their work are the representation categories of quantum groups. By performing factorization homology in this circumstance, they recover several objects related to the geometric representation theory of quantum groups.

We develop here factorization homology with target in a category of C$^*$-categories. More precisely, we take $\cS = \Clin$, the category of Karoubi complete C$^*$-categories admitting finite direct sums. This category is cocomplete (\cite{AntounVoigt}). We prove

\begin{theoremAlph}[Theorem \ref{thm:cocontinuitymaxproduct}]
The category $\Clin$ has a cocontinuous tensor structure.
\end{theoremAlph}

As a consequence, we can perform factorization homology with coefficients $\bE_2$-algebras in $\Clin$. Following standard arguments, we can also conclude

\begin{theoremAlph}[Proposition \ref{prop:EnalgebrasinClin}]
$\bE_2$-algebras in $\Clin$ are unitarily braided C$^*$-tensor categories.
\end{theoremAlph}

Many constructions in operator algebras can by now be interpreted in terms of {\em unitary tensor categories}, also known as {\em rigid C$^*$-tensor categories}. Examples include actions of certain locally compact quantum groups (\cite{Neshveyev,NeshveyevYamashitaYetterDrinfeld,HataishiYamashita}), subfactors and C$^*$-algebraic extensions (\cite{BischoffKawahigashiLongoRehrenbook,JonesPenneys1,JonesPenneys2,NeshveyevYamashitaDrinfeldCenter}). We provide evidence that C$^*$-algebraic factorization homology bridges between geometry/topology and those constructions.

Given an unitary tensor category $\cC$, denote by $\Vect(\cC)$ the category of functors $\cC^{\op} \to \Vect$. In $\Vect(\cC)$, it makes sense speaking of C$^*$-algebra objects (see Section \ref{sec:categoricalframework} and \cite{JonesPenneys1}). By establishing unitary versions of technical results in \cite{BenZviBrochierJordan1} and \cite{BenZviBrochierJordan2}, we prove

\begin{theoremAlph}[Theorem \ref{thm:factorizationhomologofpuncturedsurfaces}, Theorem \ref{thm:facthomclosedsurfaces}]
Let $\cC$ be an unitary tensor category with a unitary braiding. If $\Sigma$ is a punctured surface, there is a C$^*$-algebra object $A_\Sigma$ in $\Vect(\cC)$ such that
$$\int_{\Sigma} \cC \simeq A_\Sigma \- \Mod_\cC \ . $$
If $\Sigma$ is a closed surface, writing $\Sigma^\circ := \Sigma \setminus D$, we have
$$\int_{\Sigma} \cC \simeq \left( A_{\Sigma^\circ} \- \un_\cC \right) \- \Mod_{\int_{Ann} \cC} \ . $$
\end{theoremAlph}

Unitary tensor categories can typically be represented as categories of finite index bimodules over $\II_1$-factors, i.e. von Neumann algebras with trivial center and with a faithful tracial state. By generalizing a realization procedure in subfactor theory and combining it with factorization homology, we discuss how to produce extensions of $\II_1$-factors endowed with actions of mapping class groups. Given a C$^*$-algebra $A$ in $\Vect(\cC)$, the vector space $A(\un_\cC)$ has a structure of a C$^*$-algebra. We prove the following generalization of \cite{JonesPenneys2}. 
\begin{theoremAlph}[Theorem \ref{thm:realizationofcyclicrepresentations}]
Let $\cC \subset \Bim(N)$ the a unitary tensor category of finite index bimodules over the $\II_1$-factor $N$. There is a functorial construction taking a C$^*$-algebra object $A \in \Vect(\cC)$ and a state $\omega$ on $A(\un)$ and producing a von Neumann algebraic extension
$$N \overset{E}{\subset} \Xi(A,\omega) \ , $$
where $E: \Xi(A,\omega) \to N$ is a conditional expectation. $E$ is faithful if and only if $\omega$ is a faithful state.
\end{theoremAlph}
Given a punctured surface $\Sigma$, $A_{\Sigma}(\un_\cC)$ comes with an action of the mapping class group $\Gamma (\Sigma)$ of $\Sigma$ by $*$-automorphisms. 
 
\begin{theoremAlph}[Theorem \ref{cor:actionofmappingclassgrouponextensions}]
Let $\cC \subset \Bim(N)$ be a unitary tensor category equipped with an unitary braiding. Let $\Sigma$ be a punctured surface. Let $A_{\Sigma} \in \Vect(\cC)$ be the C$^*$-algebra object in $\Vect(\cC)$ corresponding to $\Sigma$ via factorization homology. If $\omega$ is a state on $A_{\Sigma}(\un_\cC)$ invariant under the induced action of $\Gamma(\Sigma)$, then there is an induced action of $\Gamma(\Sigma)$ on $\Xi(A_\sigma,\omega)$ such that
$$N \subset \left( \Xi(A_\sigma,\omega) \right)^{\Gamma(\Sigma)} \ . $$
That is, $N$ is globally invariant under the action of $\Gamma(\Sigma)$.
\end{theoremAlph}

For the annulus, there is a canonical invariant state: the embedding $Ann \into D$ induces, via factorization homology, a morphism of C$^*$-algebra objects $\epsilon: A_{Ann} \to \un_\cC$. In particular, $\epsilon: A_{Ann}(\un) \to \C$ is a character of the C$^*$-algebra $A_{Ann}(\un)$. The mapping class group $\Gamma(Ann)$ is isomorphic to $\Z$.

\begin{theoremAlph}[Theorem \ref{thm:invarianceofthecounitundermappingclassgroup}, Corollary \ref{cor:extensionsbyreflectioneqalgebra}]
The character $\epsilon: A_{Ann}(\un) \to \C$ is $\Gamma(Ann)$-invariant. If $\cC \subset \Bim(N)$ is an unitary tensor category with an unitary braiding, there is a von Neumann algebraic extension $N \subset \Xi(A_{Ann},\epsilon)$ such that $\Z \curvearrowright \Xi (A_{Ann},\epsilon)$ and $N$ is globally invariant under the $\Z$-action.
\end{theoremAlph}

We also analyze factorization homology with coefficients in braided fusion categories, i.e. finite unitary tensor categories. The compatibility of factorization homology with unitary structures seems to yield interesting conclusions even in this case. We remark that if $\cC$ is a fusion category of finite index bimodules over a a factor $N$ either of type $\II_1$ or properly infinite, the realization procedure is simpler: any internal C$^*$-algebra object $A \in \Vect(\cC)$ gives rise to a von Neumann algebraic inclusion $N \subset \Xi(A)$. That is, we can omit the dependency on the state, and this construction is still functorial in the appropriate sense. Moreover, it will be clear that, for any punctured surface $\Sigma$, the corresponding internal C$^*$-algebra object $A_\Sigma$ is actually an object of $\cC$

In the fusion case, we can give an interesting characterization of {\em quantum Hamiltonian reduction}, a procedure that is used to compute factorization homology on closed surfaces. For every surface $\Sigma$ there is a unique orientation preserving smooth embedding $\emptyset \into \Sigma$. This gives a canonical functor in factorization homology
$$\Hilb_{\fd} \simeq \int_{\emptyset} \cC \to \int_{\Sigma} \cC \ ,$$
where $\Hilb_{\fd}$ is the category of finite dimensional Hilbert spaces, the unit in $\Clin$. The image of $\C \in \Hilb_{\fd}$ under this functor, denoted by $\cO_\Sigma$, is called the {\em quantum structure sheaf} of $\Sigma$.

\begin{theoremAlph}[Theorem \ref{thm:vonNeumannalgebraicquantumHamiltonianreduction}]
Let $\cC \subset \Bim(N)$ be a unitarily braided fusion category of finite index bimodules over a factor $N$ either of type $\II_1$ or properly infinite. Let $\cF$ be the reflection equation algebra of $\cC$, and let $N \subset M_\cF := \Xi(\cF)$ be the corresponding von Neumann algebraic extension. Given a closed surface $\Sigma$, we have the following characterization of the endomorphism of the quantum structure sheaf of $\Sigma$:
$$ \End_{\int_\Sigma \cC} (\cO_\Sigma) \simeq \Hom_{\Bim(N)}(\cL^2(N), A_{\Sigma^\circ} \underset{M_{\cF}}{\otimes} \cL^2(N)) \ , $$
where $\Sigma^\circ$ denotes $\Sigma \setminus D$, the surface $\Sigma$ once punctured at its interior.
\end{theoremAlph}

Another perspective is the construction of topological field theories. It was shown in \cite{ScheimbauerThesis} that factorization homology can be used to construct fully extended TQFTs, making no use of the cobordism hypothesis. The TQFTs obtained from unitarily braided C$^*$-tensor categories are manifestly unitary. The present paper is a first step in the investigation of a class of C$^*$-categorical valued TQFTs, with much left to be discussed in future works.

Let us discuss the organization of the paper. In Section \ref{sec:perliminariesfachom} we describe factorization homology as developed in \cite{AyalaFrancis}, formulated in a lighter language due to the fact that our target category will be a (2,1)-category. 

Section \ref{sec:categoricalframework} is devoted to the categorical framework we are interested in. We recall the (2,1)-category of additive C*-tensor categories introduced in \cite{AntounVoigt} and study the $\underset{\max}{\boxtimes}$-symmetric monoidal product in this category, proving that it is cocontinuous. This proves that factorization homology with coefficient in an $\bE_n$-algebra in $\Clin$ is well defined. We focus later on unitary tensor categories, which are in particular rigid. We spend the remaining of Section 4 specializing the results in \cite{BenZviBrochierJordan1,BenZviBrochierJordan2} to the unitary setting. The reader familiar with their work will certainly see repetition of their arguments. However, in our case extra care must be taken since we will be dealing with categories and functors with further structure, namely $*$-structures. Most of the potential problems are solved by showing that adjoints of $*$-functors are $*$-functors, but the proof of this is to some extent non-trivial. 

In Section \ref{sec:computations} we streamline the arguments of \cite{BenZviBrochierJordan1,BenZviBrochierJordan2}, fitting them in our C*-setting. 

In Section \ref{sec:realization} we develop a theory of realization of cyclic representations for C*-algebras internal to unitary tensor categories in terms of extensions $N \subset M$ of von Neumann algebras admitting a conditional expectation, and where $N$ is a $\II_1$-factor. The strategy there is inspired by \cite{JonesPenneys2}, and our results generalize theirs. Combining the results in Section 6 with results in Section 5, we produce von Neumann algebraic extensions from geometric and topological data. 

In Section \ref{sec:quantumhamiltoninareduction}, we study quantum Hamiltonian reduction in the case of fusion categories using our formalism.  

This paper will consider the field $\C$ of complex numbers only, and hence linearity should always be understood as $\C$-linearity.

\subsection{Acknowledgements} I would like to thank A. Brochier, C. Jones, D. Penneys and M. Yamashita for their patience and kindness at discussing topics of this project with me and answering my questions.

\section{Preliminaries on Factorization Homology}
\label{sec:perliminariesfachom}

Factorization homology, as developed in \cite{AyalaFrancis,AyalaFrancisTanaka,LurieClassification}, is a topological incarnation of the notion of factorization algebras introduced in \cite{BeilinsonDrinfeldChiralAlgebras}, and elaborated in the context of perturbative QFT in \cite{CostelloOwen1,CostelloOwen2}.

Generally, factorization homology is formulated in terms of functors between higher categories. In this paper, the target of those functors will be (2,1)-categories, allowing us to simplify the exposition.

Let $M,N$ be smooth oriented surfaces. Define $\Emb(M,N)$ to be the space of oriented smooth embeddings of $M$ into $N$. Consider it as a topological space, endowed with the compact open topology. The fundamental groupoid $\Pi(\Emb(M,N))$ is then the groupoid of isotopies between smooth oriented embeddings $M \hookrightarrow N$.

\begin{definition}
Define $\Mfld^2_{\Or}$ as the (2,1)-category such that
\begin{itemize}
\item objects are smooth oriented surfaces;
\item 1-morphisms are the spaces of smooth oriented embeddings between smooth oriented surfaces;
\item 2-morphisms are isotopies between smooth oriented embeddings.
\end{itemize}
In other words, $\Mfld^2_{\Or}$ is the topological enriched category which associates to each pair $(M,N)$ of smooth oriented surfaces the topological groupoid $\Pi(\Emb(M,N))$. Together with the operation $\sqcup$ of disjoint union, $\Mfld_{\Or}^2$ becomes a symmetric monoidal (2,1)-category having the empty set $\emptyset$ as the monoidal unit.
\label{defi:Mfld}
\end{definition}

Generally, functors out of $\Mfld^2_{\Or}$ can be thought of as invariants of oriented surfaces, but one usually wishes for a compatibility with respect to disjoint unions, i.e., compatibility with respect to the monoidal structure, and also some formula allowing the computation of the invariant of a manifold, when it is obtained by gluing other two, in terms of the invariants associated to each manifold in the gluing, i.e., an excision formula. Factorization homology will give a recipe to construct such functors.

Suppose $(\cS,\boxtimes)$ is a symmetric monoidal (2,1)-category whose objects are certain linear categories, and the 2-morphisms are natural isomorphisms, and suppose that $F: \Mfld^2_{\Or} \to (\cS, \boxtimes)$ is a monoidal functor. Let $D \subset \R^2$ be the disk. Observe that the value $\cC := F(D)\in \cS$ comes equipped with certain operations $ \cC^{\boxtimes n} \to \cC$ which are images of the oriented smooth embeddings $D^{\sqcup n} \into D$ under $F$. Isotypic embeddings are sent to natural isomorphic functors. Such a structure is known as an $\bE_2$-algebra in $\cS$. In the following, $\Disk^2_{\Or}$ is the full monoidal subcategory of $\Mfld^2_{\Or}$ generated by the disk $D$. Being singly generated as a monoidal category, any strict monoidal functor out of $\Disk^2_{\Or}$ is determined by its value at $D$.

\begin{definition}
Let $\cS$ be a symmetric monoidal (2,1)-category. An $\bE_2$-algebra in $\cS$ is a monoidal functor $F: \Disk^2_{\Or} \to \cS$. Generally, an $\bE_n$-algebra in a symmetric monoidal category $\cS$ is a monoidal functor $\Disk_n^{\Or} \to \cS$.
\label{defi:E2-algebra}
\end{definition}

With some abuse of language, $F(D)$ itself will be referred to as an $\bE_2$-algebra. 

Let $F: \Mfld^2_{\Or} \to \cS$ be a monoidal functor. Then $\cC = F(D)$ is an $\bE_2$-algebra. Given an oriented surface $M$, $F(M)$ has the following property: in the figure below, any commutative diagram as in the left side induces through $F$ a commutative diagram as in the right side:

\begin{center}
\begin{tikzcd}
(D)^{\sqcup n} \arrow[rr, hook] \arrow[rd, hook] & {} \arrow[loop, distance=2em, in=305, out=235] & (D)^{\sqcup m} \arrow[ld, hook] &  & \cC^{\boxtimes n} \arrow[rr] \arrow[rd] & {} \arrow[loop, distance=2em, in=305, out=235] & \cC^{\boxtimes m} \arrow[ld] \\
                                                    & M                                              &                                    &  &                                         & F(M)                                           &                               
\end{tikzcd}
\end{center}
If $\cS$ is cocomplete, then $F(M)$ is the universal object having this property. In this case, one usually writes
$$F(M) = \underset{(D)^{\sqcup n} \into M}{\colim} \cC^{\boxtimes n} =: \int_M \cC \ , $$
the integral notation at the right-hand-side indicating that $F$ can be computed via a ``passage from local-to-global" procedure. Thus, when $\cC$ is cocomplete, $F$ is completely determined by its value at $D$. 

\begin{definition}
In the above, $\int_{(-)} \cC = F$ is called the factorization homology with coefficients in $\cC$ and with values in $\cS$.
\label{defi:factorizationhomology}
\end{definition}

In what follows we shall write $I$ for the oriented unit interval. Let $M$ be a surface obtained by a collar gluing of the surfaces $M_1$ and $M_2$ along a 1-dim. manifold $M_0$: 
$$M \simeq M_1 \underset{M_0 \times I}{\sqcup} M_2$$
There is an embedding $M_0 \times I \sqcup M_0 \times I \into M_0 \times I$ by stacking cylinders. By functoriality, there is an induced functor
$$\int_{M_0 \times I} \cC \ \boxtimes \int_{M_0 \times I} \cC \to \int_{M_0 \times I} \cC \ . $$
This gives a monoidal structure to the factorization homology over $M_0 \times I$.

Also, there are embeddings 
$$M_1 \sqcup M_0 \times I \into M_1 \ \text{and}\  M_0 \times I \sqcup M_2 \into M_2 \  ,$$
providing factorization homology over $M_1$ and $M_2$ with a right module structure and a left module structure over $\int_{M_0 \times I} \cC$, respectively. A monoidal functor $F = \int_{(-)} \cC$ as above is said to have the excision property with respect to collar gluing if
$$\int_M \cC \simeq \int_{M_1} \cC \underset{\int_{\small M_0 \times I} \cC}{\boxtimes} \int_{M_2} \cC \ . $$

\begin{theorem}[\cite{AyalaFrancis}]
If $(\cS,\boxtimes)$ is a cocomplete symmetric monoidal (2,1)-category and $\boxtimes$ is cocontinuous, then for any $\bE_2$-algebra $\cC$ in $\cS$, factorization homology with coefficient $\cC$ is well defined and has the excision property with respect to collar gluing.
\label{thm:excision}
\end{theorem}

\section{C*-categorical Framework}
\label{sec:categoricalframework}

In this Section we establish the categorical framework for C$^*$-algebraic factorization homology. We shall write $\Vect$ for the category of vector spaces and linear maps, $\Vect_{\fd}$ for its full subcategory of finite dimensional vector spaces, $\Hilb$ for the category of Hilbert spaces and bounded linear maps and $\Hilb_{\fd}$ for its full subcategory of finite dimensional Hilbert spaces.  

\subsection{C*-categories}

A linear category is a small category $\cC$ enriched over $\Vect$: for any pair $(X,Y)$ of objects in $\cC$, $\Hom_{\cC}(X,Y)$ is a vector space, and the composition of morphisms is bilinear. A $*$-structure on a linear category $\cC$ is a contravariant endofunctor $*: \cC \to \cC$ which at the level of objects is the identity map $\Obj(\cC) \to \Obj(\cC)$ and which is involutive at the level of morphisms: $* \circ *: \Hom_\cC(X,Y) \to \Hom_\cC (Y,X) \to \Hom_\cC(X,Y)$ is the identity for all $X,Y \in \Obj(\cC)$. $*$-structures are also known as {\em dagger-structures}. Write $X \in \cC$ if $X \in \Obj(\cC)$. The definition below follows \cite{AntounVoigt} and \cite{JonesPenneys1}.

\begin{definition}
A $*$-category $\cC$ is a C*-category if
\begin{itemize}
\item[(1)] for every $X,Y \in \cC$, every $f \in \Hom_\cC(X,Y)$, there exists $g \in \End_\cC(X)$ such that $f^* \circ f = g^* \circ g$;
\item[(2)] for every $X,Y \in \cC$, the function $\| \cdot \|: \Hom_\cC(X,Y) \to [0,+\infty)$ given by
$$\| f \|^2 := \sup \{ |\lambda| \geq 0 \ | \nexists (\ f^* \circ f - \lambda \id_X)^{-1} \} $$  
makes $\Hom_{\cC}(X,Y)$ into a Banach space, with the additional properties $\| g \circ f \| \leq \|g\| \|f\|$ for every pair $(f,g)$ of composable morphisms, and $\| f^* \circ f \| = \|f\|^2$ for all morphisms $f$.
\end{itemize}
\label{defi:C*-categories}
\end{definition}

In particular, a C*-category is a category enriched over the category of Banach spaces, where the morphisms in the latter category are the contractive linear maps, such that every endomorphism algebra $\End_\cC(X)$ is a C*-algebra. We shall consider only categories with direct sums. In this case, condition (1) in the above definition is redundant, since it can be derived from the C*-identity together with an amplification trick.

In a C*-category $\cC$, a morphism $f$ is said to be unitary if $f^* = f^{-1}$. If $f$ is an endomorphism, it is said to be self-adjoint if $f^* = f$ and a projection if $f = f^* = f^2$. A morphism $f$ such that $ff^*$ is a projection is called a partial isometry, and in this case it also follows that $f^*f$ is a projection. Given two C*-categories $\cC$ and $\cC'$, a $*$-functor from $\cC$ to $\cC'$ is a linear functor $F: \cC \to \cC'$ which commutes with the respective $*$-structures: $F(f)^* = F(f^*)$ for every morphism $f$ in $\cC$. If $F$ and $F'$ are $*$-functor between C*-categories, a natural transformation $\eta: F \to F'$ is called a unitary natural isomorphism if every component $\eta_X \in \Hom_{\cC'}(F(X),F'(X))$ is unitary.

\begin{definition}
Define $\Clin$ to be the (2,1) category consisting of
\begin{itemize}
\item objects: C*-categories with direct sum;
\item 1-morphisms: $*$-functors;
\item 2-morphisms: unitary natural isomorphisms between functors.
\end{itemize}
\label{defi:C*-lin}
\end{definition}

\begin{remark}
In \cite{AntounVoigt}, the authors make a distinction between categories with or without units, but wee shall deal only with unital categories.
\end{remark}

\subsection{The monoidal structure of $\Clin$}

In \cite{AntounVoigt} it is shown that  $\Clin$ is cocomplete, i.e., it has small colimits. It has at least two monoidal products $\underset{\min}{\boxtimes}$ and $\underset{\max}{\boxtimes}$, induced by the min and max tensor products of C*-algebras, respectively. The main result of this section is that $\underset{\max}{\boxtimes}$ is cocontinuous. This monoidal structure is an operator algebraic version of the Deligne-Kelly monoidal product.

Let $\cC$ and $\cD$ be C*-categories with direct sums. We first introduce an intermediate category $\cC \boxtimes \cD$ and then define $\cC \underset{\max}{\boxtimes} \cD$ to be its Karoubi completion. Objects in $\cC \boxtimes \cD$ are pairs $(X,Y) \in \Obj(\cC) \times \Obj(\cD)$; they will be denoted by $X \boxtimes Y$. Define
$$\End_{\cC \boxtimes \cD} (X \boxtimes Y) := \End_\cC(X) \underset{\max}{\otimes} \End_\cD(Y) \ , $$
the $\max$ tensor product of the C*-algebras $\End_\cC(X)$ and $\End_\cD(Y)$. Given objects $X \boxtimes Y$ and $X' \boxtimes Y'$, there is an inclusion
$$\Hom_\cC(X,X') \underset{\C}{\otimes} \Hom_\cD(Y,Y') \into \End_{\cC \boxtimes \cD}((X \boxtimes Y) \oplus (X' \boxtimes Y')) \ . $$
Then $\Hom_{\cC \boxtimes \cD}(X \boxtimes Y, X' \boxtimes Y')$ is defined as the closure of the image of this inclusion. Direct sums of objects are defined coordinatewisely and extended to morphisms in the obvious way.

\begin{definition}
Define $\cC \underset{\max}{\boxtimes} \cD$ to be the Karoubi completion of $\cC \boxtimes \cD$.
\end{definition}

Now we move towards showing that $\underset{\max}{\boxtimes}$ is cocontinuous. We follow the terminology in \cite{AntounVoigt}.

Let $I$ be a small category, and let $V = (V_i)_{i \in I}$ be an $I$-shaped diagram in $\Clin$. That is, $V: I \to \Clin$ is a functor. The colimit $\varinjlim V_i$ is characterized by the equivalence
\begin{align*}
[I,\Clin] ( (V_i)_{i \in I}, \Delta W )  \simeq \Clin (\varinjlim V_i, W) 
\end{align*}
for all $W \in \Clin$, where $[I, \Clin]$ denotes the category of $I$-shaped diagrams in $\Clin$ and $\Delta W$ is the constant $I$-shaped diagram with essential value  $W$.

\begin{lemma}
For $A,B \in \Clin$, let $A \boxdot B: \Clin \to \Cat$ be given by 
$$ (A \boxdot B) (C) := \Clin (A, \Clin(B,C)) \ .  $$
Then $A \boxdot B$ is represented by $A \underset{\max}{\boxtimes} B$, i.e., 
$$\Clin (A \underset{\max}{\boxtimes}B,C) \simeq \Clin ( A, \Clin(B,C)) \ \forall \ C \ . $$
\end{lemma}

\begin{proof}
This is Proposition 3.8 in \cite{AntounVoigt}; we will give just a complementary remark. Given $F: A \to \Clin(B,C)$ and objects $(a,b) \in A \times B$, $F$ induces  $*$-homomorphisms 
$$\End_A(a) \overset{\pi_A}{\rightarrow} \End_C(F_a(b)) \overset{\pi_B}{\leftarrow} \End_B(b) \ ,$$
given respectively by 
$$\pi_A(\phi) = F_\phi(b) \ , \ \pi_B(\psi) = F_a(\psi) \ .$$ 
By definition, $F_a(\phi)$ is a natural endomorphism of the functor $F_a: B \to C$. Naturality then implies 
$$F_\phi(b) \circ F_a(\psi) = F_a(\psi) \circ F_\phi(b) \ , $$
which in turn implies, together with the universality of the max tensor product, the existence and uniqueness of a $*$-homomorphism
$$ \End_A(a) \underset{\max}{\boxtimes} \End_B(b) \to \End_C(F_a(b)) \ ,$$
which restricts to $\pi_A$ and $\pi_B$. From this construction, we obtain a functor $f: A \underset{\max}{\boxtimes} B \to C$.
\end{proof}

\begin{corollary}
There are equivalences of C$^*$-categories
\begin{align*}
\Clin (A \underset{\max}{\boxtimes} \varinjlim V_i, B) \simeq \Clin(A, \Clin(\varinjlim V_i,B))) \simeq \Clin( A, [I,\Clin]((V_i)_{i \in I}, \Delta B)
\end{align*}
\end{corollary}

\begin{lemma}
There is an equivalence of C*-categories
$$ \Clin (A,[I,\Clin]((V_i)_i, \Delta B)) \simeq [I,\Clin ] ((A \underset{\max}{\boxtimes} V_i)_i, \Delta B) \ . $$
\end{lemma}

\begin{proof}
Let $F \in  \Clin (A,[I,\Clin]((V_i)_i, \Delta B))$. Then for every $a \in A$ there is a family $\{F(a)_i: V_i \to B\}_i$ of linear $*$-functors satisfying relations dictated by the category $I$. Linearity of $F$ means that, for all $i$,
$$F( - )_i: A \times V_i \to B$$
is bilinear, and hence induces a linear $*$-functor $\tilde{F}_i: A \underset{\max}{\boxtimes} V_i \to B$, and it must be shown that $\{\tilde{F}_i\}$ is a transformation from $(A \underset{\max}{\boxtimes} V_i)_i$ to $\Delta B$. Write $(V_i)_{i \in I} = \iota: I \to \Clin$, i.e., $V_i = \iota(i) \ \forall \ i \in  I$. The $I$-shaped diagram $(A \underset{\max}{\boxtimes} V_i)_{i \in I}$ is then $\id_A \boxtimes \iota$. $\{\tilde{F}_i \}_{i \in I}$ is a transformation of $I$-shaped diagrams iff for all morphisms $i \to j$ in $I$ it holds
$$ \tilde{F}_i = \tilde{F}_j \circ (\id_A \boxtimes \iota (i \to j)) \ .$$ 
Observe that
\begin{align*}
\tilde{F}_i &= \tilde{F}_j \circ( \id_A \boxtimes \iota (i \to j)) \\
& \Updownarrow \\ 
\tilde{F}_i( a \boxtimes -) 
& = \tilde{F}_j (a \boxtimes - ) \circ (\id_A \boxtimes \iota (i \to j) ) \\
& \Updownarrow \\
F(a)_i & =  F(a)_j \circ \iota (i \to j)
\end{align*}
for all $a \in A$, the last condition being one of the defining conditions for $F$.

The correspondence $F \mapsto \tilde{F}$ gives a map $\theta$ from the objects of  $ \Clin (A,[I,\Clin]((V_i)_i, \Delta B))$ to the objects of $[I,\Clin ] ((A \underset{\max}{\boxtimes} V_i)_i, \Delta B)$. Given two objects $F$ and $G$ in the former category, assume that $\Gamma: F \to G$ is a natural transformation. By definition, $\Gamma(a): F(a) \to G(a)$ is a modification of transformations for all $a \in A$. Define $\tilde{\Gamma}: \tilde{F} \to \tilde{G}$ by 
$$\tilde{\Gamma}_i ( a \boxtimes v) := \Gamma(a)_i (v) \ . $$
The commutativity of
\begin{center}
\begin{tikzcd}
\tilde{F}_i \arrow[rr, "\tilde{\Gamma}_i"] \arrow[dd, "\tilde{F}(i \to j)"']     &                                                & \tilde{G}_i \arrow[dd, "\tilde{G} (i \to j)"]       \\
                                                                                 & {} \arrow[loop, distance=2em, in=305, out=235] &                                                     \\
\tilde{F}_j \circ (\id_A \boxtimes \iota(i \to j)) \arrow[rr, "\Gamma_j * \id"'] &                                                & \tilde{G}_j \circ (\id_A \boxtimes \iota (i \to j))
\end{tikzcd}
\end{center}
is equivalent to the commutativity of
\begin{center}
\begin{tikzcd}
F(a)_i \arrow[rr, "\Gamma(a)_i"] \arrow[dd, "F(a)(i \to j)"'] &                                                & G(a)_i \arrow[dd, "G(a) (i \to j)"] \\
                                                              & {} \arrow[loop, distance=2em, in=305, out=235] &                                     \\
F(a)_j \circ \iota (i \to j) \arrow[rr, "\Gamma(a)_j * \id"'] &                                                & G(a)_j \circ \iota(i \to j)        
\end{tikzcd}
\end{center}
for all $a \in A$, which is a defining condition for $\Gamma$. This finishes the construction of the functor $\Theta: \Clin (A, [I,\Clin]((V_i)_{i \in I}, \Delta B))\to [I,\Clin]((A \underset{\max}{\boxtimes} V_i)_{i \in I}, \Delta B)$. Reversing the arguments, given $\tilde{F} \in [I, \Clin]((A \underset{\max}{\boxtimes} V_i)_{i \in I}, \Delta B)$, one sees that, for each $a \in A$,
$$\{ F(a)_i:= \tilde{F}_i( a \boxtimes -): V_i \to B \}_{i \in I} $$
defines an element of $[I,\Clin]((V_i)_{i \in I}, \Delta B)$. Naturality of $\tilde{F}$ implies that $F$ extends to a functor
$$F: A \to [I,\Clin]((V_i)_{i \in I}, \Delta B) \ . $$
The construction of a natural transformation $\Gamma: F \to G$ from a modification $\tilde{\Gamma}: \tilde{F} \to \tilde{G}$ follows the same reasoning.
\end{proof}

\begin{theorem}
The monoidal structure $\underset{\max}{\boxtimes}$ on $\Clin$ is cocontinuous.
\label{thm:cocontinuitymaxproduct}
\end{theorem}

\begin{proof}
The $\varinjlim (A \underset{\max}{\boxtimes} V_i)$ is determined up to equivalence by
$$\Clin (\varinjlim A \boxtimes V_i, B) \simeq [I, \Clin] ((A \boxtimes V_i)_{i \in I}, \Delta B ) $$
for all $B \in \Clin$. Then, using the previous lemmas, one concludes
\begin{align*}
\Clin (\varinjlim(A \underset{\max}{\boxtimes} V_i), B) & \simeq [I, \Clin] ((A \underset{\max}{\boxtimes} V_i)_{i \in I}, \Delta B))  \\
& \simeq \Clin( A, [I,\Clin]((V_i)_{i \in I}, \Delta B)) \\
& \simeq \Clin ( A, \Clin( \varinjlim V_i, B)) \\
& \simeq \Clin ( A \boxtimes \varinjlim V_i, B)
\end{align*}
\end{proof}

By \cite{AyalaFrancis}, factorization homology is well defined in $\Clin$. See Definition \ref{defi:C*-tensor categories} for a definition of C$^*$-tensor categories.

\begin{proposition}
In $\Clin$, $\bE_1$-algebras are C*-tensor categories, $\bE_2$-algebras are unitarily braided C*-tensor categories, and for $n \geq 3$ $\bE_n$-algebras are unitarily symmetric C*-tensor categories.
\label{prop:EnalgebrasinClin}
\end{proposition}

\begin{proof}
The analogous statement in $\Cat(\C)$, the symmetric category of linear categories, is well known. The proof of the Proposition is identical to that case. For every $n \geq 1$, an embedding of two disjoint $n$-dimensional disks into a third one gives to any $\bE_n$-algebra $\cC$ a monoidal structure. For $n = 2$, there is a nontrivial orientation preserving isotopy that exchanges the positions of the two small disks inside the bigger one. This corresponds to a 2-morphisms in $\Clin$ between $\cC \underset{\max}{\boxtimes} \cC \overset{\otimes}{\longrightarrow} \cC$ and $\cC \underset{\max}{\boxtimes} \cC \overset{\text{flip}}{\longrightarrow} \cC \underset{\max}{\boxtimes} \cC \overset{\otimes}{\longrightarrow} \cC$, which is exactly a braiding. Since 2-morphisms in $\Clin$ are unitary natural isomorphisms, the braiding must be unitary. For $n \geq 3$, the isotopy exchanging the positions of the two small disks is trivial, meaning that the braiding must be symmetric.
\end{proof}

\begin{definition}
Given a linear category $\cA$, an object $X \in \cA$ is said to be compact iff the functor $\Hom_\cA(X, -): \cC \to \Vect$ preserves filtered colimits. The category $\cA$ is compactly generated iff any one of its objects is a filtered colimit of compact objects. 
\end{definition} 

The terminology in the above definition agrees with \cite{BenZviBrochierJordan1}. In \cite{AntounVoigt}, compact objects are called finitely presentable.

\begin{definition}
The free cocompletion of a category $\cA$ is the category $\Vect(\cA)$ of linear functors $\cA^{\op} \to \Vect$ and natural transformations.
\label{defi:algebraicindcompletion}
\end{definition}

\begin{remark}
When $\cA$ is semisimple, the free cocompletion is equivalent to the ind-completion, which is the cocompletion of $\cA$ with respect to filtered colimits. We shall soon restrict our attention to unitary tensor categories, and those are automatically semisimple. See Definition \ref{defi:C*-tensor categories} and Remark \ref{rem:UTCaresemisimple}. 
\end{remark}

The Yoneda Lemma says that $\cA$ is identified with a full subcategory of $\Vect(\cA)$ by means of $X \mapsto \Hom_\cA( - , X)$. A linear category $\cA$ is compactly generated iff it is equivalent to the algebraic ind-completion of its full subcategory of compact objects.

\begin{definition}
Let $\Clin^K$ be the (2,1)-category of Karoubi complete finitely additive C*-categories, 1-morphisms being linear $*$-functors and 2-morphisms being unitary natural isomorphisms.
\label{defi:KaroubicompleteC*categories}
\end{definition}

Notice that $\Clin^K$ is a cocomplete full subcategory of $\Clin$. By construction, the max tensor product $\underset{\max}{\boxtimes}$ restricts to a tensor product on $\Clin^K$. Therefore, factorization homology with coefficients in $\bE_2$-algebras in $\Clin^K$ is well defined and has the excision property. For simplicity, we shall often write $\boxtimes$ instead of $\underset{\max}{\boxtimes}$.

A distinguished class of objects in $\Clin$ will be of importance in this work: the C*-tensor categories;

\begin{definition}
A C*-tensor category in $\Clin^K$ is a C*-category $\cC$ in $\Clin^K$ with a monoidal structure $(\otimes, \un, \alpha, \lambda, \rho)$, tensor product, unit object, associator, left and right unit structure morphisms, respectively, such that $\otimes: \cC \times \cC \to \cC$ is a $*$-functor and the structure morphisms are all unitary. An {\em unitary tensor category} (UTC) is a rigid C*-tensor category (see \cite{NeshveyevTuset}).
\label{defi:C*-tensor categories}
\end{definition}

\begin{remark}
If $\cC$ is an UTC, it can be shown that $\End_\cC(X)$ is a finite dimensional C$^*$-algebra for all $X \in \cC$. Every finite dimensional C$^*$-algebra is a direct sum of full matrix algebras. From this one can deduce that $\cC$ is semisimple.
\label{rem:UTCaresemisimple}
\end{remark}

\begin{example}
Let $N$ be a factor either of type $\II_1$ or type $\III$. By an $N \- N$ bimodule we shall mean a Hilbert space $\cH$ together with two commuting normal actions $N \curvearrowright \cH \curvearrowleft N$. Given two $N \- N$-bimodules $\cH_1$ and $\cH_2$, let $\Hom_{N \- N}(\cH_1,\cH_2)$ be the space of bounded $N \- N$-bimodular linear maps from $\cH_1$ to $\cH_2$. The category $\Bim(N)$ of $N\-N$-bimodules with morphism spaces the spaces $\Hom_{N \- N}(\cH_1,\cH_2)$ can be endowed with a monoidal structure, the Connes fusion of bimodules $(\cH_1, \cH_2) \mapsto \cH_1 \underset{N}{\otimes} \cH_2$, making $\Bim(N)$ a C*-tensor category. There is a full subcategory $\Bim_0(N)$ of $\Bim(N)$, consisting of the bimodules having finite left and right coupling constant, also known as Jones index. Then $\Bim_0(N)$ is an unitary tensor category. 
\end{example} 

\begin{example}
Let $N$ be a type $\III$ factor. Denote by $\End(N)$ the monoid of normal $*$-endomorphisms of $N$. Defining, for $\theta,\rho \in \End(N)$, 
$$\Hom(\theta,\rho) = \{x \in N \ | \ x \theta(n) = \rho(n) x  \ \forall \ n \in N \} \ , $$
$\End(N)$ becomes a C*-tensor category: the tensor product of endomorphisms is given by composition, while for morphisms $x \in \Hom(\theta,\rho)$ and $y \in \Hom(\theta',\rho')$,
$$x \otimes y := x \theta(y) \in \Hom(\theta \circ \theta', \rho \circ \rho') \ . $$

It follows from the spatial theory of A. Connes that, as C*-tensor categories, $\End(N) \simeq \Bim(N)$ (see \cite{LongoII}). The full subcategory $\End_0(N)$ corresponding to the spherical $N \- N$-bimodules is then an unitary tensor category.

\end{example}

For more details on the above example, we refer the reader to \cite{AnantharamanPopa, BischoffKawahigashiLongoRehrenbook,LongoII}.

\subsection{Module categories}

If $\cM$ is a linear category, the endofunctor category $\End(\cM)$, objects being endofunctors and morphims natural transformations, has a canonical structure of monoidal category: the tensor product is given by composition, the unit is the identity functor $\id_\cM$, and all the structure morphisms are strict, i.e., they are all identity morphisms. When $\cM$ is a C*-category, we denote by $\End(\cM)$ the C*-category of $*$-endofunctors of $\cM$, with completely bounded natural transformations as morphisms. More precisely, a natural transformation $\eta: F \to F'$ between two $*$-functors is determined by a natural collection $( \eta_X \in \Hom_\cM(F(X),F'(X)))_{X \in \cM}$, and we ask for
$$\sup_X \| \eta_X \| < \infty \ . $$

\begin{definition}
Let $\cC$ be a C*-tensor category. A right $\cC$-module category is a category $\cM$ with a tensor functor $\cC^{\op} \to \End(\cM)$. If $\cM$ is a C*-category and the functor $\cC \to \End(\cM)$ is a $*$-tensor functor, we say that $\cM$ is a $\cC$-module C*-category.	
\label{defi:modulecategories}
\end{definition}

Let $(\cM, F:\cC^{\op} \to \End(\cM))$,  be a $\cC$-module C*-category as in the above definition. We shall write
$$F(X)(m) =: m \triangleleft X \ , \ \forall m \in \cM, \ \forall \ X \in \cC \ , $$
and will often omit $F$. Then the unitary tensor structure of $F$ says that there are natural unitary isomorphisms
\begin{align*}
(m \triangleleft X) \triangleleft Y & \simeq m \triangleleft (X \otimes Y) \ , \\
m \triangleleft \un_\cC & \simeq m \ ,
\end{align*}
with further compatibility conditions.

Here is a way of constructing module categories. Given an algebra object $A \in \Vect(\cC)$, let $\tilde{\cM}_A$ be the category of free left $A$-modules: it has $\{ A \otimes X \ | \ X \in \cC \} \subset \Vect(\cC)$ as objects, and the morphisms are $A$-linear morphisms. Let $\cM_A$ be the Karoubi completion of $\tilde{\cM}_A$. Then
$$(A \otimes X) \triangleleft Y := A \otimes X \otimes Y $$
extends to a right $\cC$-module structure on $\cM_A$.

There is a converse construction: an object $m \in \cM$ determines a functor $\Hom_\cM(m \triangleleft (-),m): \cC^{\op} \to \Vect$, i.e., an element in the algebraic ind-completion $\Vect(\cC)$ of $\cC$.

\begin{definition}
The internal endomorphism algebra of $m \in \cM$ in $\cC$, defined uniquely up to isomorphism when it exists, is an object $\underline{\End}_\cC(m)$ in $\Vect(\cC)$ representing the functor $\Hom_\cM (m \triangleleft (-), m)$. That is, there is a family of natural isomorphisms
$$\Hom_\cM (m \triangleleft X, m) \simeq \Hom_{\Vect(\cC)}(X, \underline{\End}_\cC(m) )$$
indexed by $X \in \cC$. Internal endomorphism algebras are always algebra objects in $\Vect(\cC)$.
\label{defi:internalend}
\end{definition}

\begin{proposition}[\cite{Ostrik}]
If $\cC$ and $\cM$ are semisimple, internal endomorphism algebras always exist.
\end{proposition}

This is so because the functor $X \mapsto \Hom_\cM (m \triangleleft X, m)$ is exact.

In the semisimple case the two constructions presented above, passing from algebras to module categories and back, are not equivalences. The module categories obtained from algebra objects are generated, as module categories, by one element. One says that a module category $\cM$ with a generating object is cyclic, and when restricted to such, there is an equivalence:

\begin{proposition}[\cite{Ostrik}]
Suppose $\cC$ is semisimple. Then the category $\Alg(\Vect(\cC))$ of algebra objects in $\Vect(\cC)$ is equivalent to the category $\{(\cM,m)\}$ of cyclic $\cC$-module categories with a distinguished generating object. For $A \in \Alg(\Vect(\cC))$, the corresponding cyclic module category is the Karoubi completion $(\cM_A,A)$ of the category of free $A$-modules. Starting with a cyclic module category $(\cM,m)$, the associated algebra object is the internal endomorphism $\underline{\End}(m)$. 
\label{prop:equivalencebetweenalgebrasandmodules}
\end{proposition}

\begin{definition}
Let $\cC$ be an UTC. An algebra $A \in \Alg(\Vect(\cC)$ is said to be a C*-algebra object if the corresponding $\cC$-module category is a $\cC$-module C*-category. Write $\CAlg(\Vect(\cC))$ for the category of C*-algebra objects in $\Vect(\cC)$, with $*$-algebra maps as morphisms.
\label{defi:Cstaralgebraobjects}
\end{definition}

C*-algebras objects are discussed more thoroughly in \cite{JonesPenneys1}. See also \cite{HataishiYamashita} for an intrinsic definition, without reference to module categories.

The next theorem might be known among experts in C*-categories, but I could not find a proof of it anywhere.

\begin{theorem}
Suppose that $F: \cC \to \cD$ is a $*$-functor between C*-categories admitting a right adjoint $F^R$. Then $F^R$ is automatically a $*$-functor.	
\label{prop:adjointofastarfunctor}
\end{theorem}

\begin{proof}
In any $*$-category $\cA$, the involution provides anti-linear isomorphisms $\Hom_\cA(X,X') \simeq \Hom_\cA(X',X)$, for all $X,X' \in \cA$. This is quite a strong feature about $*$-categories, which implies for instance that if a $*$-functor $F: \cC \to \cD$ admits a right adjoint $F^R$, then $F^R$ is also a left adjoint for $F$. Indeed, if the the adjunction $F \leftrightarrow F^R$ is witnessed by a natural transformation $\Psi = \{ \Psi_{X,Y}: \Hom_\cD(F(X),Y) \to \Hom_\cC(X,F^R(Y)) \} $, then $\Phi := \{ \Phi_{Y,X} :  \Hom_\cC(F^R(Y),X) \to \Hom_\cD(Y,F(X))\}$ defined via the commutativity of the diagram
\begin{center}
\begin{tikzcd}
{\Hom_\cC(F^R(Y),X)} \arrow[rr, "{\Phi_{Y,X}}"] \arrow[d, "*"'] &  & {\Hom_\cD(Y,F(X))}                 \\
{\Hom_\cC(X,F^R(Y))} \arrow[rr, "{\Psi_{X,Y}^{-1}}"']           &  & {\Hom_\cD(F(X),Y)} \arrow[u, "*"']
\end{tikzcd}
\end{center}
witnesses an adjunction $F^R \leftrightarrow F$. From this fact we shall deduce the Theorem.

The strategy to prove that $F^R$ is a $*$-functor goes as follows. Since $\cC$ is a C*-category, the space $\Hom_\cC(X,F^R(Y))$ has a canonical structure of a right Hilbert C*-module over $\End_\cC(X)$, where the action is given by $f \triangleleft a := f \circ a$, and the $\cC(X)$-valued inner product is given by $\left\langle f,g \right\rangle := f^* \circ g$. We define a Hilbert C*-module, over $\End_\cC(X)$, on $\Hom_\cD(F(X),Y)$ sucht that $\Psi_{X,Y}$ is a isomorphism of Hilbert C*-modules. Then a morphism $T \in \Hom_\cD(Y,Y')$ defines an adjointable linear map between the Hilbert C*-modules $\Hom_\cD(F(X),Y)$ and $\Hom_\cD(F(X),Y')$, which corresponds to the adjointable linear map $F^R(T) \circ (-): \Hom_\cC(X,F^R(Y)) \to \Hom_\cC(X,F^R(Y'))$. The Theorem is then obtained by observing that both $F^R(T)^* \circ (-)$ and $F^R(T^*) \circ (-)$ are adjoints of this map, and therefore must coincide. Let us check this.

The $\End_\cC(X)$-action of $\Hom_\cD(F(X),Y)$ is given by
$$f \triangleleft a := \Psi_{X,Y}^{-1}(\Psi_{X,Y}(f) \circ a ) \ , $$
and the $\End_\cC(X)$-valued inner product is given by
$$\left\langle f, g \right\rangle := \Psi_{X,Y}(f)^* \circ \Psi_{X,Y}(g) \ . $$ 
By construction, $\Psi_{X,Y}$ is an isomorphism of Hilbert C*-modules. In particular there is an identification
$$\cL_{\cC(X)}(\cD(F(X),Y),\cD(F(X),Y')) = \Psi_{X,Y'}^{-1} \cL_{\cC(X)} (\cC(X,F^R(Y)),\cC(X,F^R(Y')) \Psi_{X,Y} \ . $$

Given $T \in \cD(Y,Y')$ the map $F^R(T) \circ (-): \cC(X,F^R(Y)) \to \cC(X,F^R(Y'))$ is adjointable, the adjoint given by $F^R(T)^* \circ (-)$. From the above equality, $T \circ (-): \cD(F(X),Y) \to \cD(F(X),Y')$ is adjointable. One has
\begin{align*}
\left\langle T \circ f, g \right\rangle &= \Psi_{X,Y'} (T \circ f)^* \circ \Psi_{X,Y'}(g) \\
&= \Psi_{X,Y'}((T \circ f)^{**})^* \Psi_{X,Y'}(g) \\
&= \Phi_{Y',X}(f^* \circ T^*) \Psi_{X,Y'}(g) \\
& = \Phi_{Y,X}(f^*) F^R(T^*) \Psi_{X,Y'}(g) \\
& = \Psi_{X,Y}(f)^* \Psi_{X,Y}(T^* \circ g) \ ,
\end{align*}
by naturality of $\Phi$. The last expression is $\left\langle f, T^* \circ g \right\rangle$, showing that $F^R(T^*) \circ (-)$ is the adjoint of $F^R(T) \circ (-)$.
\end{proof}

Putting together Proposition \ref{prop:adjointofastarfunctor} and Proposition 3.11 in \cite{BenZviBrochierJordan1}, we obtain the following.

\begin{corollary}
Let $\cC$ be an unitary tensor category, $\cM$ and $\cN$ be $\cC$-module C$^*$-categories. If $F: \cM \to \cN$ is a $\cC$-module $*$-functor admitting a right adjoint $F^R$, then $F^R: \cN \to \cM$ has a canonical structure of a $\cC$-module $*$-functor.
\end{corollary}

\subsection{Reconstruction Theorems}

In \cite{BenZviBrochierJordan1,BenZviBrochierJordan2}, the authors explicitly computed factorization homology for punctured, closed and marked surfaces, through monadic reconstruction theorems. Those theorems provided them with canonical identifications between linear categories, allowing the computation of the topological invariants. In this subsection we will obtain analogous results in $(\Clin^K, \underset{\max}{\boxtimes})$. Many of the arguments here follow the same reasoning as in \cite{BenZviBrochierJordan1,BenZviBrochierJordan2}. Some, however, are approached differently.

Let $\cC$ be an unitary tensor category. As stated in Proposition \ref{prop:equivalencebetweenalgebrasandmodules}, cyclic $\cC$-module categories can be described internally to $\cC$. Recall in particular that if $\cM$ is a cyclic $\cC$-module C*-category, with distinguished generator $m$, then $\underline{\End}_\cC(m)$ is a C*-algebra in $\Vect(\cC)$, and there is an equivalence of $\cC$-module C*-categories
$$\cM \simeq \underline{\End}_\cC \- \Mod_\cC \ . $$
The following is the C*-version of Theorem 4.6 in \cite{BenZviBrochierJordan1}.

\begin{lemma}
Let $\cM$ and $\cN$ be $\cC$ -module C*-categories in $\Clin^K$, and let $F: \cM \to \cN$ be a $\cC$-module $*$-functor. Then, for every $m \in \cM$, there is a canonical $*$-homomorphism of C*-algebras
$$\rho_F: \underline{\End}_\cC (m) \to \underline{\End}_\cC(F(m)) \ . $$
\label{lemma:functorialityformonads}
\end{lemma}

\begin{proof}
Let $A := \underline{\End}(m)$, and let $B:= \underline{\End}(F(m))$. Consider the $\cC$-module C*-subcategories $\cM_A$ and $\cN_B$ of $\cM$ and $\cN$, respectively, generated as modules by $A$ and $B$, respectively. $F$ restricts to a linear $\cC$-module $*$-functor $F: \cM_A \to \cN_B$, which corresponds to a $*$-natural transformation $\rho_F$ between the $C^*$-algebra objects $A$, and $B$, see \cite{JonesPenneys1}. Concretely, $\rho_F$ is constructed as follows: let $\Irr(\cC)$ be a choice of maximal family of pairwise non-isomorphic irreducible objects of $\cC$. Writing $A(X) := \cM_A(A \triangleleft X, A)$ and $B(X) := \cN_B (B \triangleleft X,B)$, one has
$$A \simeq \bigoplus_{X \in \Irr(\cC)} A(X) \otimes X \ \ \text{and} \ \ B \simeq \bigoplus_{X \in \Irr(\cC)} B(X) \otimes X \ ,$$
and $\rho_F$ is induced by the maps $F_{X,\un} \otimes \id_X: \cM_A(A \triangleleft X, A) \otimes X \to \cN_B( B \triangleleft X, B) \otimes X$.
\end{proof}

In case $\cM$ and $\cN$ are cyclic with respective generating objects $m$ and $n$, then $F^R: \cN \to \cM$ is naturally isomorphic to the pull-back functor along $\rho_F$.

\begin{lemma}
Right cyclic $\cC$ -module C*-categories are left-dualizable, while left cyclic $\cC$-module C*-categories are right dualizable.
\label{dualizabilityofmodulecategories}
\end{lemma}

\begin{proof}
Let $A$ be a $C^*$-algebra object in $\cC$, and consider its associated cyclic $\cC$-module C*-category $\cM_A$. Associated to $A$ is also a cyclic left $\cC$-module C*-category $\tensor[_A]{\cM}{}$, generated under idempotent completion by the category of free right $A$-modules in $\cC$. We shall show that $\tensor[_A]{\cM}{}$ dualizes $\cM_A$. Let the coevaluation $c: \Vect \to \cM_A \underset{\cC}{\boxtimes}  \tensor[_A]{\cM}{}$ be defined by
$$c(\C) := (A \otimes \un) \underset{\cC}{\boxtimes} (\un \otimes A) \ , $$
and the evaluation $e: \tensor[_A]{\cM}{} \boxtimes \cM_A \to  \cC$ be given by
$$ e( (U \otimes A) \boxtimes (A \otimes V)) := U \otimes V \ . $$
Then the compostion
\begin{center}
\begin{tikzcd}
\cM_A \simeq \Vect \boxtimes \cM_A \arrow[rr, "c \boxtimes \id_{\cM_A}"] &  & (\cM_A \underset{\cC}{\boxtimes}  \tensor[_A]{\cM}{}) \boxtimes \cM_A \simeq \cM \underset{\cC}{\boxtimes} \left( \tensor[_A]{\cM}{} \boxtimes \cM_A \right) \arrow[rr, "\id_{\cM_A} \boxtimes e"] &  & \cM_A \underset{\cC}{\boxtimes} \cC \simeq \cM_A
\end{tikzcd}
\end{center}
is equal to the identity functor, as witnessed by the following computation:
\begin{align*}
A \otimes U &\mapsto \left( (A \otimes \un) \boxtimes_\cC (\un \otimes A)  \right) \boxtimes (A \otimes U) \\
& \simeq (A \otimes \un) \underset{\cC}{\boxtimes} \left( (\un \otimes A) \boxtimes (A \otimes U) \right) \\
&\mapsto (A \otimes \un) \underset{\cC}{\boxtimes} U \\
&\simeq A \otimes U \ .
\end{align*}
Both $e$ and $c$ are $*$-linear functors. Therefore, $\tensor[_A]{\cM}{}$ is the dual of $\cM_A$ as a $\cC$-module C*-category.
\end{proof}

\begin{lemma}
Let $A$ be a C*-algebra object in $\Vect(\cC)$ and let $(\cM_A,A)$ be the right cyclic $\cC$-module C*-category of left $A$-modules in $\cC$. If $\cN$ is a left $\cC$-module C*-category, not necessarily cyclic, then
$$\cM_A \underset{\cC}{\boxtimes} \cN \simeq A \hyphen \Mod_\cN \ .$$
\label{reconstructionforrelativetensorproduct}
\end{lemma}

\begin{proof}
Let $\Fun_\cC(\tensor[_A]{\cM}{}, \cN)$ be the category of linear $\cC$-module $*$-functors from $\tensor[_A]{\cM}{}$ to $\cN$. Observe first that
$$\cM_A \underset{\cC}{\boxtimes} \cN \simeq \Fun_\cC(\tensor[_A]{\cM}{}, \cN) \ .$$
Indeed, the equivalence is implemented by the functors $F_{(\cdot)}: \cM_A \underset{\cC}{\boxtimes} \cN \to \Fun_\cC(\tensor[_A]{\cM}{}, \cN)$, wich sends $x$ to the functor
\begin{center}
\begin{tikzcd}
F_x: _A\cM \arrow[rr, "\id \boxtimes x"] &  & \tensor[_A]{\cM}{} \boxtimes \cM_A \underset{\cC}{\boxtimes} \cN \arrow[rr, "e \boxtimes \id"] &  & \cN
\end{tikzcd},
\end{center} 
and its quasi-inverse $x_{( \cdot)}: \Fun_\cC(\tensor[_A]{\cM}{}, \cN) \to \cM_A \underset{\cC}{\boxtimes} \cN$, which sends the functor $F$ to the image $x_F$ of $\C$ under the composition
\begin{center}
\begin{tikzcd}
\Vect \arrow[rr, "e"] &  & \cM_A \underset{\cC}{\boxtimes} \tensor[_A]{\cM}{} \arrow[rr, "\id \boxtimes F"] &  & \cM_A \underset{\cC}{\boxtimes} \cN
\end{tikzcd}.
\end{center}
The last morphism in the above composition only makes sense for $\cC$-module functors $F$.

Now, let $G: \cM_A \underset{\cC}{\boxtimes} \cN \to \Fun_\cC(\tensor[_A]{\cM}{},\cN) \to \cN$, where the first arrow is the equivalence just described and the second is the evaluation $e_A$ at $A \in \ \tensor[_A]{\cM}{}$. $G$ has a left adjoint $G^L$, given by $G^L(n) = A \underset{\cC}{\boxtimes} n$. Both $G$ and $G^L$ are $*$-linear functors. Since $A$ generates $\tensor[_A]{\cM}{}$, the evaluation map is conservative, i.e., $e_A(F) = F(A) = 0 \iff F = 0$. Consequently, $G$ is also conservative. The Barr-Beck monadicity theorem then implies
$$\cM_A \underset{\cC}{\boxtimes} \cN \simeq A \hyphen \Mod_\cN \ , $$
linearly. The equivalence is implemented by the functor
$$\cM_A \underset{\cC}{\boxtimes} \cN \ni A \underset{\cC}{\boxtimes} n \mapsto A \triangleright n \ , $$
which is a unitary functor, and thus the linear equivalence is also an unitary equivalence.
\end{proof}

\begin{corollary}
Suppose $\cD$ is another $C^*$-tensor category and that $T: \cC \to \cD$ is a unitary tensor functor. If $\cM_A$ is a cyclic right $\cC$-module C*-category, then
$$ \cM_A \underset{\cC}{\boxtimes} \cD \simeq T(A)\hyphen \Mod_\cD \ . $$
\label{monadicityforbasechange}
\end{corollary}
\begin{proof}
In the above expression, $\cD$ is regarded as a right $\cC$-module, the action being induced by $T$. The claim is then a direct consequence of the previous lemma.
\end{proof}

It follows from Corollary \ref{monadicityforbasechange} that $\cM_A \underset{\cC}{\boxtimes} \cD$ can be endowed with a canonical left $\cD$-module structure, which will have $A \underset{\cC}{\boxtimes} \un_\cD \simeq T(A)$ as cyclic generator. Specializing to the case when $T: \cC \to \cD$ is dominant, we get the following Lemma.

\begin{lemma}
Let $\cM$ be a $\cC$-module C*-category with a generator $m$. If $F: \cC \to \cD$ is a dominant $*$-tensor functor, then $m \underset{\cC}{\boxtimes} \un_\cD$ is a generator for the $\cD$-module category $\cM \underset{\cC}{\boxtimes} \cD$, and it holds
$$\cM \underset{\cC}{\boxtimes} \cD \simeq F( \underline{\End}(m)) \hyphen \Mod_\cD \ . $$
\label{lemma:monadicityforbasechange}
\end{lemma}

So far we have considered the categorical constructions which will be used in the computation of the factorization homology for punctured surfaces. Now we move on to the constructions needed in the computation of the factorization homology for closed surfaces. This is the content of \cite{BenZviBrochierJordan2}, which we adapt to our C*-framework. 

\subsection{Reconstruction for tensor functors}

Recall that a tensor functor $F:\cC \to \cB$ between tensor categories endows the target category $\cB$ with the structure of both a left and a right $\cC$-module category. When $\cC$ is an UTC an $\cB$ is a C$^*$-tensor category, the internal endomorphism algebra $\underline{\End}(\un_{\cB})$ comes with an additional structure, that of an Yetter-Drinfeld C$^*$-algebra object in $\cC$. The terminology is reminiscent from the theory of quantum group actions. In Appendix \ref{app:YetterDrinfeldC*algebras} we present a systematic discussion of such structures, elaborating on the idea that such Yetter-Drinfeld C$^*$-algebra objects are C$^*$-algebra objects in $\cZ(\Vect(\cC))$.

Now we deduce an unitary analogue of Theorem 4.1 in \cite{BenZviBrochierJordan2} by means of an unitary version of Proposition 5.1 in \cite{BruguieresNataleExactsequences}. 

\begin{theorem}
Let $\cC$ be an unitary tensor category, not necessarily braided. Suppose that $F: \cC \to \cB$ is an unitary tensor functor to a C*-tensor category $\cB$ such that $\un_\cB$ is a generator for the induced $\cC$-action. Then $M_B := \underline{\End}(\un_\cB)$ has a half-braiding which makes it a commutative C*-algebra object in $\cZ \left( \Vect(\cC) \right)$ (see Appendix \ref{app:YetterDrinfeldC*algebras} and Definition \ref{def:Yetter-DrinfeldC*algebra}). Moreover, there is an equivalence of C*-tensor categories
$$\cB \simeq \Bim_\cC(\underline{\End}(\un_\cB)) \ , $$
where $\Bim_\cC(\underline{\End}(\un_\cB))$ denotes the Karoubi completion of the C$^*$-category of $\underline{\End}(\un_{\cB})$-bimodules of the form $\underline{\End}(\un_{\cB}) \otimes U \otimes \underline{\End}(\un_{\cB})$, with $U \in \cC$.

\label{thm:reconstructionfortensorcategories}
\end{theorem}

\begin{proof}
Since $F$ is an unitary tensor functor, 
$$X \triangleright \un_\cB = F(X) \otimes \un_\cB \simeq F(X) \simeq \un_\cB \otimes F(X) = \un_\cB \triangleleft F(X) \ .$$
By \cite{HataishiYamashita}, $M_B = \underline{\End}(\un_B)$ is a C*-algebra object in $\Vect(\cC)$ with a unitary half-braiding $\beta = \{ \beta_X: M_B \otimes X \simeq X \otimes M_B \}_{X \in \cC}$ (see Definition \ref{def:unitaryhalfbraiding}) which we now describe. $F^R$ has a canonical structure of a $\cC$-bilinear module functor. Note that $F$ is unitarily equivalent to the functor $act_{\un_B}: X \mapsto \un_\cB \triangleleft X = \un_\cB \otimes F(X)$, and thus $F^R \simeq act_{\un_\cB}^R$. Taking all this into account, $\beta$ is computed as follows:
\begin{align*}
M_B \otimes X &= act_{\un_\cB}^R(\un_\cB) \otimes X \\
& \simeq act_{\un_\cB}^R( \un_\cB \triangleleft X) \\
& = act_{\un_\cB} \otimes F(X) \\
& \simeq act_{\un_\cB}^R(F(X) \otimes \un_\cB) \\
& = act_{\un_\cB}^R(X \triangleright \un_\cB) \\
& \simeq X \otimes act_{\un_\cB} (\un_\cB) \\
& = X \otimes M_B \ \
\end{align*}
where all the isomorphisms are canonical, unitary and natural in $X$.

The multiplication  $m: M_B \otimes M_B \to M_B$ is provided by the adjunction data between $act_{\un_\cB}$ and $act_{\un_\cB}^R$:
\begin{align*}
M_B \otimes M_B &= act_{\un_\cB}^R(\un_\cB) \otimes act_{\un_\cB}^R(\un_\cB) \\
& \simeq act_{\un_\cB} (\un_\cB \triangleleft act_{\un_\cB}^R(\un_\cB)) \\
& = act_{\un_\cB}(act_{\un_\cB} act_{\un_\cB}^R (\un_\cB)) \overset{act_{\un_\cB}^R(\epsilon_{\un_\cB})}{\longrightarrow} act_{\un_\cB}^R(\un_\cB) = M_B \ ,
\end{align*}
where $\epsilon: act_{\un_\cB} \circ act_{\un_\cB}^R \to \id_\cB$ is the counit of the adjunction. 

To show that $M_B$ is commutative is to show that $m = m \circ \beta_{M_B}$. If $\cB$ was strict, the equality would be tautological since we would have
$$act_{\un_\cB}^R(\un_\cB) \otimes act_{\un_\cB}^R(\un_\cB) \simeq act_{\un_\cB}^R(\un_\cB \triangleleft act_{\un_\cB}^R(\un_\cB)) = act_{\un_\cB}^R( act_{\un_\cB}^R(\un_\cB) \triangleright \un_\cB) \ . $$
Since every monoidal category is equivalent to a strict one, the claim about the commutativity of $M_B$ follows.

If $X$ is a left $M_B$-module in $\Vect(\cC)$, define $X \otimes M_B \to X$ as the composition
$$X \otimes M_B \overset{\beta_X^{-1}}{\to} M_B \otimes X \to X \ . $$
It defines a right $M_B$-module structure on $X$. The commutativity of $M_B$ implies that the left and right $M_B$ actions commute, so that $X$ becomes an $M_B$-bimodule. One has $M_B \- \Mod_\cC \simeq M_B \- \Bim_\cC$, and the composition
$$\cB \simeq M_B\- \Mod_\cC \simeq M_B \- \Bim_\cC$$
becomes an equivalence of  C*-tensor categories, where in $M_B \- \Bim_\cC$ the monoidal structure is given by
$$ X \boxtimes Y \to X \underset{M_B}{\otimes} Y \ ,$$
and the unit is $M_B$. 
\end{proof}

\begin{remark}
When $\cC = \Rep_{\fd} (\bG)$ for a compact quantum group $G$, the authors in \cite{NeshveyevYamashitaYetterDrinfeld} identified the category unitary tensor functors $F: \Rep_{\fd}(G) \to \cB$ having $\un_\cB$ as a generator with the category of unital {\em braided-commutative} Yetter-Drinfeld $G$-C*-algebras. As a corollary of Theorem \ref{thm:reconstructionfortensorcategories} and its proof, we conclude that in this case the category of unital braided commutative $G$-C*-algebras is equivalent to the category of C*-algebras in $\Vect(\cC)$ equipped with a half-braiding making it a commutative C*-algebra in $\cZ(\Vect(\cC))$.
\end{remark}

\begin{definition}
A Yetter-Drinfeld C*-algebra $A \in \Vect(\cC)$ is braided-commutative if it is commutative in $\cZ(\Vect(\cC))$.
\label{defi:braidedcommutalgebras}
\end{definition}

 The following proposition and also its proof are unitary versions of Proposition 3.7 in \cite{SafronovMomentMaps}.

\begin{proposition}

Let $\cC, \cB, M_B = \underline{\End}_\cC(\un_\cB)$ and $\beta$ be as in Theorem \ref{thm:reconstructionfortensorcategories}. There is an equivalence between the category of C*-algebras in $\Vect(\cB)$ and the category of C*-algebras $A$ in $\Vect(\cC)$ equipped with a $*$-algebra homomorphism $\mu: M_B \to A$ such that
\begin{center}
\begin{tikzcd}
M_B \otimes A \arrow[r, "\mu \otimes \id"] \arrow[dd, "\beta_A"'] & A \otimes A \arrow[rd, "m_A"]  &   \\
                                                                  &                                & A \\
A \otimes M_B \arrow[r, "\id \otimes \mu"']                       & A \otimes A \arrow[ru, "m_A"'] &  
\end{tikzcd}
\end{center}
commutes.
\label{prop:monadicreconstructionthroughtensorfunctor}
\end{proposition}

\begin{proof}
The category of C*-algebras in $\Vect (\cB)$ is equivalent to the category of C*-algebras in $\Vect(M_\cB \- \Bim_\cC$), which is in turn equivalent to the category of C*-algebras in  $\Vect( M_\cB \- \Mod_\cC)$. In this latter category the monoidal structure is the one introduced in the proof of Theorem \ref{thm:reconstructionfortensorcategories}. A C*-algebra $A$ in $\Vect( M_\cB \- \Mod_\cC )$ is in particular a C*-algebra in $\cC$. Since the unit object in $M_\cB \- \Mod_\cC$ is $M_\cB$, the unit map of $A$ is a $*$-homomorphism $\mu: M_\cB \to A$. Now, the functor $M_\cB \- \Mod_\cC \into M_\cB \- \Bim_\cC$ implementing the equivalence is the one which endows every left $M_\cB$-module with the right $M_\cB$-module structure using $\beta$. This means exactly the commutativity of the above diagram.
\end{proof}

Using again that the adjoint of a $*$-functor is itself a $*$-functor, we obtain the C*-version of Theorem 4.3 in \cite{BenZviBrochierJordan2}. Let $\cM$ be a $\cB$-module C*-category. If $m \in \cM$ is a generator for the induced $\cC$ -action, so that $\cM \simeq \underline{\End}(m) \- \Mod_\cC$, by Proposition \ref{prop:monadicreconstructionthroughtensorfunctor} there is a $*$-algebra homomorphism $\rho: \underline{\End}(\un_\cB) \to \underline{\End}_\cB(m)$.

\begin{theorem}
The $\cB$ action on $\cM$, under the above identifications, is given by
$$\cM \underset{\max}{\boxtimes} \cB \ni n \boxtimes b \mapsto n \underset{\underline{\End}(\un_\cB)}{\otimes} b \ .$$
The generator $m$ under the $\cC$ -action is also a generator under the $\cB$-action, and the functor $\cM \simeq \underline{\End}_\cB (m) \- \Mod_\cB \to \underline{\End}_\cC(m) \- \Mod_\cC$ is an equivalence of $\cB$-module C*-categories. 
\label{thm:canonicalformofmodulesovertheannuluscategory}
\end{theorem}

Applying  Theorem \ref{thm:canonicalformofmodulesovertheannuluscategory} and Lemma \ref{lemma:functorialityformonads}, we obtain the C*-version of Corollary 4.4 in \cite{BenZviBrochierJordan2}:

\begin{corollary}
If $\cM$ is a right $\cB$-module C*-category with a generator $m$ and $\cN$ is a left $\cB$-module C*-category with a generator $n$, wrt. the $\cC$-actions, then there is an equivalence of C*-categories
$$\cM \underset{\cB}{\boxtimes} \cN \simeq \left( \underline{\End}(m) \- \underline{\End}(n) \right) \- \Bim_{\underline{\End}(\un_\cB)} \ . $$
\label{cor:balancedtensorproductovertheannuluscategory}
\end{corollary}

Suppose now that $\cC$ has an unitary braiding $\sigma$. We shall apply the previous results later in the case $\cB = \int_{Ann} \cC =: U(\cC)$. Here, the unitary tensor functor $\cC \to U(\cC)$ will be reminiscent of a stronger structure, which allows a strengthening of the above constructions.

\section{Computations in Factorization Homology}
\label{sec:computations}

In this section we streamline the computations in \cite{BenZviBrochierJordan1, BenZviBrochierJordan2}, pointing out the compatibility with the unitary structures.

\begin{remark}
In the purely algebraic setting of linear categories, one has to distinguish between factorization homology for framed surfaces and factorization homology for oriented surfaces. This is so because for the latter the $\bE_2$-algebra needs an extra structure, that of a ribbon/balancing structure.
Every braided unitary tensor category has a canonical ribbon/balancing structure (\cite{NeshveyevTuset}), so a framed theory canonically descends to an oriented theory in our framework.
\end{remark}

\subsection{Punctured surfaces}

Punctured surfaces are obtained by collar gluing of handles into a disk along marked intervals on the boundary of the the disk. The gluing is described by a combinatorial data called gluing pattern, a bijection $P: \{1,1',...,n,n'\} \to \{1,2,...,2n\}$, from which the surface is obtained as follows. Take n handles with ends indexed as $\{1,1',...,n,n'\}$ such that $(i,i')$ are the ends of the $i$-th handle. On the boundary of a disk, mark 2n intervals, indexed by $\{1,...,2n\}$. To obtain the associated surface $\Sigma_P$, glue the ends $(i,i')$ of the $i$-th handle to the disk along the respective intervals $(P(i),P(i'))$.

{\bf Examples:}
\begin{itemize}
\item[(1)] The gluing pattern $P:\{1,1'\} \to \{1,2\}$ associated to the annulus is given by $P(1) = 1$ and $P(1') = 2$.

\begin{figure}[h]
\centering
\label{fig:gluingpatternannulus}
\includegraphics[scale=0.3]{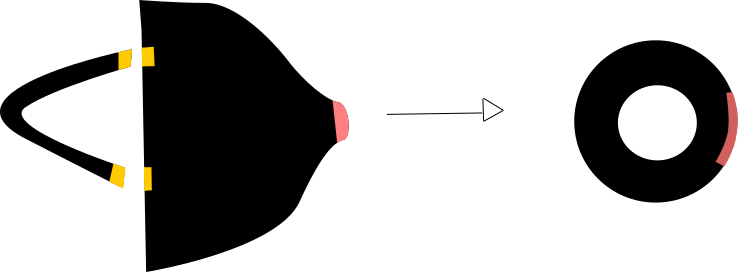}
\caption{Gluing pattern associated to the annulus}
\end{figure}

\item[(2)] The gluing pattern $P:\{1,1',2,2'\} \to \{1,2,3,4\}$ associated to the punctured torus is given by $P(1) = 1$, $P(1') = 3$, $P(2) = 2$ and $P(2') = 4$.
\begin{figure}[h]
\centering
\label{fig:gluingpatterntorus}
\includegraphics[scale=0.2]{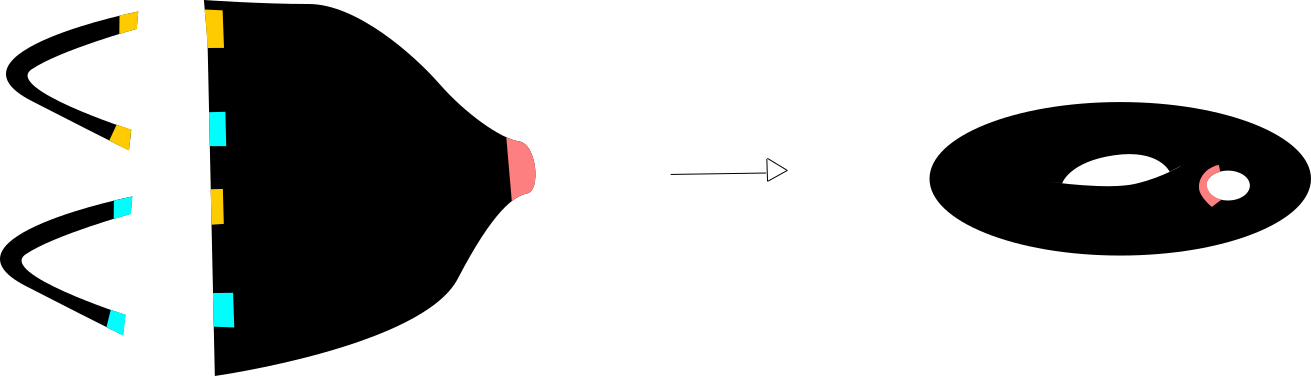} 
\caption{Gluing pattern associated to the punctured torus}
\end{figure}
\end{itemize}

We will assume from now on that $\cC \in \Clin^K$ is a unitary tensor category, therefore rigid and semisimple. Remember the pointed structure of factorization homology: given an oriented surface $\Sigma$, the unique oriented embedding $\emptyset \into M$ induces a functor $\Hilb_{\fd} \to \int_M \cC$. The quantum structure sheaf $\cO_\Sigma$ of $\Sigma$ is the image of $\C$ under this functor.

If $P$ is a gluing pattern, the presence of a boundary component in $\Sigma_P$ gives to $\int_{\Sigma_P} \cC$ a $\cC$-module structure. It is induced by a choice of a collar in the boundary of $\Sigma_P$ and an embedding $\Sigma_P \sqcup D \to \Sigma_P$ where $\Sigma_P$ is embedded outside the collar and $D$ is embedded in the interior of the collar. In figures 1 and 2 we illustrate the choice of such a collar. 

Let $A_P = \underline{\End}_\cC (\cO_{\Sigma_P})$, the internal endomorphism of the quantum structure sheaf wrt. the action of $\cC$. We shall prove, in analogy to \cite{BenZviBrochierJordan1}, that there is a canonical equivalence of $\cC$-module C*-categories
$$\int_{\Sigma_P} \cC \simeq A_P \- \Mod_\cC \ .$$

We start by computing factorization homology for the annulus. An unitary braiding on $\cC$ gives to it  left and right $\cC^{\boxtimes n}$ actions, which we will call regular actions. The braiding being unitary implies that the tensor functors  $\cC^{\boxtimes n} \to \End(\cC)$ are unitary tensor functors. The monoidal unit $\un$ is a generator for the regular actions. A related observation is that, if $T: \cC \boxtimes \cC \to \cC$ is the monoidal structure on $\cC$, an unitary braiding on $\cC$ gives to $T$ an unitary tensor structure. Notice that $\cC \boxtimes \cC$ is also an UTC, and as a $\cC \boxtimes \cC$-module C$^*$-category, $(\cC, \un)$ is cyclic. Hence the internal endomorphism $\underline{\End}_{\cC \boxtimes \cC}(\un)$ exists, and $\cC \simeq \underline{\End}_{\cC \boxtimes \cC}(\un) \- \Mod_{\cC \boxtimes \cC}$ as a module.

Consider the gluing pattern of the annulus. The handle $H$ is isomorphic to the disk, but the marking of the interval at the ends of it gives to $\int_H \cC \simeq \cC$ a right $\cC \boxtimes \cC$-module structure. Similarly, the marking of the two intervals on the disk along which the collar gluing is performed (Figure 1) gives to $\int_{D} \cC \simeq \cC$ a left $\cC \boxtimes \cC$-module structure. Both of them are induced by $T$, and from the gluing pattern one concludes
$$\int_{Ann}\cC \simeq \cC \underset{\cC \boxtimes \cC}{\boxtimes} \cC \ . $$

\begin{definition}
The reflection equation algebra of $\cC$ is defined to be
$$\cF := T( \underline{\End}_{\cC^{\boxtimes 2}}(\un)) \ . $$
\label{reflectionequationalgebra}
\end{definition}

A simple computation, using the cocontinuity of the functor $X \mapsto \Vect(\cC) (X, \cF)$ shows that 
$$\underline{\End}_{\cC^{\boxtimes 2}}(\un) \simeq \bigoplus_{X \in \Irr(\cC)} \bar{X} \boxtimes X \ , $$ 
Notice that, even though the direct sum above is an algebraic direct sum over a possibly infinite set, the resulting $*$-algebra is indeed a $C^*$-algebra object: for each $U \boxtimes V \in \cC \boxtimes \cC$, there are finitely many $X \in \Irr(\cC)$ for which $\cC \boxtimes \cC (U \boxtimes V, \bar{X} \boxtimes X)$ is non-zero.

The monoidal unit $\un_\cC$ is a generator for the $\cC^{\boxtimes 2}$-action,
$$\cC_{\cC \boxtimes \cC} \simeq \underline{\End}_{\cC^{\boxtimes 2}}(\un) \hyphen \Mod _{\cC^{\boxtimes 2}} \ . $$
Since $T$ is a dominant tensor functor, there is an equivalence
$$\cC \underset{\cC \boxtimes \cC}{\boxtimes} \cC \simeq T(\underline{\End}_{\cC^{\boxtimes 2}}(\un)) \hyphen \Mod_\cC = \cF \hyphen \Mod_\cC $$
of $\cC$-module C*-categories.

\begin{proposition}
Let $\cC$ be an UTC equipped with an unitary braiding. Then
$$ \cF = \bigoplus_{X \in \Irr(\cC)}^{c_0} \bar{X} \otimes X \ , $$
i.e., as a functor $\cF: \cC^{\op} \to \Vect(\cC)$, we have
$$\cF(U) = \bigoplus_{X \in \Irr(\cC)}^{c_0} \cC(U, \bar{X} \otimes X) \ . $$
\end{proposition}

\begin{proof}
Let us denote by $\boxdot$ the Deligne-Kelly tensor product of linear dagger categories, and let us write $\cF_{alg}$ for the algebraic version of the reflection equation algebra; it is the $*$-algebra object in $\Vect(\cC)$ representing
$$\cC \underset{\cC \boxdot \cC}{\boxdot} \cC$$
as a cyclic $\cC$-module dagger category. The ground $*$-algebra $\cF_{alg}(\un)$ has a canonical C$^*$-norm given by
$$\|R_X\| = d_X^{1/2} \ ,$$
Passing from $\boxdot$ to $\boxtimes = \underset{\max}{\boxtimes}$ corresponds to taking the completion $\cF(\un)$ of $\cF_{alg}(\un)$ wrt. to this norm, and completing each fiber $\cF_{alg}(X)$ to a Hilbert C$^*$-module over $\cF(\un)$. Such completions are given exactly by the $c_0$-direct sums.
\end{proof}

We shall use the reflection equation algebra to compute factorization homology for general punctured surfaces. We say that a gluing pattern $P:\{1,1',...,n,n'\} \to \{1,...,2n\}$ has rank $n$.

\begin{definition}
Let $P$ be a gluing pattern of rank $n$. The $i$-th and $j$-th handles, $H_i$ and $H_j$ respectively, $1 \leq i < j \leq n$, are
\begin{itemize}
\item positively linked $\iff P(i) < P(j) < P(i') < P(j')$;
\item negatively linked $\iff P(j) < P(i) < P(j') < P(i')$;
\item positively nested $\iff P(i) < P(j) < P(j') < P(i')$;
\item negatively nested $\iff P(j) < P(i) < P(i') < P(j')$;
\item unlinked $\iff$ $H_i$ and $H_j$ are neither linked nor nested.
\end{itemize}
\end{definition}

Consider the following automorphisms of $\cF \otimes \cF$: for $U$ and $V$ in $\Irr(\cC)$, let $L_{U,V}, N_{U,V}, U_{U,V}: \bar{U} \otimes U \otimes \bar{V} \otimes V$ be given by
\begin{align*}
L_{U,V} &:= (\iota \otimes \sigma_{V,\bar{U}}^* \otimes \iota ) \circ ( \sigma_{\bar{U},\bar{V}} \otimes \sigma_{U,V}) \circ (\iota \otimes \sigma_{U,\bar{V}}\otimes \iota) \\
N_{U,V} &:= (\iota \otimes \sigma_{V, \bar{U}}^* \otimes \iota) \circ (\sigma_{\bar{V},\bar{U}}^* \otimes \sigma_{U,V}) \circ (\iota \otimes \sigma_{U,\bar{V}} \otimes \iota) \\
U_{U,V} &:= \sigma_{\bar{U} \otimes U, \bar{V} \otimes V} 
\end{align*}
They induce, respectively, automorphisms
$$L = \bigoplus_{U,V} L_{U,V} \ , \ N = \bigoplus_{U,V} N_{U,V} \ \text{and} \ U = \bigoplus_{U,V} U_{U,V} $$
of $\cF \otimes \cF$. They will be associated with linked, nested and unlinked handles of the gluing patterns, respectively.

Let $P$ be a gluing pattern of rank $n$, and let $a_P := \cF^{\otimes n}$. Let $\cF^{(i)}$ be the image of $\cF$ under the embedding $\cF \hookrightarrow a_P$ into the i-th factor. Consider the function $C:\N_n \times \N_n \to \{L^{\pm 1}, N^{\pm 1}, U\}$ given by $C_{i,j} = L^{\pm1 }$ if $H_i$ and $H_j$ are positively / negatively linked, $C_{i,j} = N^{\pm 1}$ if they are positively / negatively nested and $C_{i,j} = U$ if they are unlinked. Give to $a_P$ the algebra structure such that the embeddings $\cF \into \cF^{(i)}$ are algebra maps and, when restricted to $\cF_\cC^{(i)} \otimes \cF_\cC^{(j)}$, it restricts to
\begin{center}
\begin{tikzcd}
\cF_\cC^{(i)} \otimes \cF_\cC^{(j)} \otimes \cF_\cC^{(i)} \otimes \cF_\cC^{(j)}  \arrow[rr, "\iota \otimes C_{ij} \otimes \iota"] &  & \cF_\cC^{(i)} \otimes \cF_\cC^{(i)} \otimes \cF_\cC^{(j)} \otimes \cF_\cC^{(j)}  \arrow[rr, "m \otimes m"] &  & \cF_\cC^{(i)} \otimes \cF_\cC^{(j)} 
\end{tikzcd}
\end{center}

By construction, if $P_1$ and $P_2$ are two gluing patterns and $P_1 \sqcup P_2$ is their disjoint union, then
$$a_{P_1 \sqcup P_2} \simeq a_{P_1} \boxtimes a_{P_2} \ . $$
is the braided tensor product of $a_{P_1}$ and $a_{P_2}$ (see Appendix \ref{app:Braidedtensorproduct}).
\begin{theorem}
For a gluing pattern $P$ and its associated punctured surface $\Sigma_P$,
\begin{align*}
\int_{\Sigma_P} \cC \simeq a_P \- \Mod_\cC \ .
\end{align*}
\label{thm:factorizationhomologofpuncturedsurfaces}
\end{theorem}

\begin{proof}
Consider the disk $D$ marked with 2n intervals on the left and one marked interval on the right. These markings induce the $\cC^{\boxtimes 2n} - \cC$ bimodule structure on $\cC = \int_D \cC$ given by tensoring on the corresponding sides. On each handle $H_i$, the factorization homology gives $\cC$ with the right $\cC^{\boxtimes 2}$-module structure also giving by tensoring on the right. All actions are implemented by unitary tensor functors. Thus,
$$\cC^P := \int_{\bigsqcup_{i=1}^n H_i} \cC $$
is the $C^*$-category $\cC^{\boxtimes n}$ with a $\cC^{\boxtimes 2n}$-module structure given by
$$(a_1 \boxtimes ... \boxtimes a_n) \triangleleft (b_1 \boxtimes ... \boxtimes b_{2n}) = (a_1 \otimes b_{P(1)} \otimes b_{P(1')}) \boxtimes ... \boxtimes (a_n \otimes b_{P(n)} \otimes b_{P(n')}) \ . $$
By the excision property of factorization homology,
$$\int_{\Sigma_P} \cC \simeq \cC^P \underset{{\cC^{\boxtimes 2n}}}{\boxtimes} \cC \ . $$

There is a permutation $\tau_P \in S_{2n}$ such that the $\cC^{\boxtimes 2n}$-module structure of $\cC^P$ is given by
\begin{center}
\begin{tikzcd}
\cC^P \boxtimes \cC^{\boxtimes 2n} = \cC^{\boxtimes n} \boxtimes \cC^{\boxtimes 2n} \arrow[rr, "\iota \boxtimes \tau_P"] \arrow[dd, "\triangleleft"'] &                                                & \cC^{\boxtimes n} \boxtimes \cC^{\boxtimes 2n} \arrow[dd, "\iota \boxtimes T_\cC^{\boxtimes n}"] \\
                                                                                                     & {} \arrow[loop, distance=2em, in=305, out=235] &                                                                                                  \\
 \cC^P = \cC^{\boxtimes n}                                                                                    &                                                & \cC^{\boxtimes n} \boxtimes \cC^{\boxtimes n} \arrow[ll, "T_{\cC^{\boxtimes n}} "]               
\end{tikzcd},
\end{center} 
and in terms of this permutation, 
$$\cC^P \simeq \tau_P \left( \underline{\End}_{\cC \boxtimes \cC}(\un)^{\boxtimes n} \right) \- \Mod_{\cC^{\boxtimes 2n}} \ . $$
Let $\tilde{a}_P = \tau_P \left( \underline{\End}_{\cC \boxtimes \cC}(\un)^{\boxtimes n} \right)$. By Lemma \ref{reconstructionforrelativetensorproduct}, 
$$\int_{\Sigma_P} \cC \simeq \cC^P \underset{\cC^{\boxtimes 2n}}{\boxtimes} \cC \simeq T^{2n}(\tilde{a}_P) \- \Mod_\cC \ .$$

$$T^{2n} (\tilde{a}_P) \simeq \bigoplus_{X_1,...,X_n \in \Irr(\cC)}^{c_0} \bar{X}_{\tau_P(1)} \otimes X_{\tau_P(1')} \otimes ... \otimes \bar{X}_{\tau_P(n)} \otimes X_{\tau_P(n')} \ . $$

Let $\cF^{(i,i')}$ denote the image of  $\underline{\End}_{\cC^2} (\un_\cC)$ inside $T^{2n}(\tilde{a}_P)$ under the map which places $\bar{X}_i$ and $X_i$ in the factors $\tau_P^{-1}(i)$ and $\tau_P^{-1}(i')$, respectively. Then $\cF^{(i,i')}$, for $1 \leq i \leq n$, are subobjects of $\tilde{a}_P$, and to determine the algebra structure $m$ of $T^{2n}(\tilde{a}_P)$ is to determine the multiplication between the different factors $T^{2n}(\cF^{(i,i')})$ and $T^{2n}(\cF^{(j,j')})$. From the unitary tensor structure of $T: \cC \boxtimes \cC \to \cC$ and from the multiplication law of $\cF^{\otimes n}$, it follows that
\begin{center}
\begin{tikzcd}
                                                                  & {T^{2n}(\cF^{(i,i')} \otimes \cF^{(j,j')}) = T^{2n}( \cF^{(j,j')} \otimes \cF^{(i,i')})} \arrow[dd, "m"] &                                                                   \\
\cF^{(i)} \otimes \cF^{(j)} \arrow[ru, "C_{ij}"] \arrow[rd, "m"'] &                                                                                                         & \cF^{(j)} \otimes \cF^{(i)} \arrow[lu, "C_{ji}"'] \arrow[ld, "m"] \\
                                                                  & T^{2n}(\tilde{a}_P)                                                                                     &                                                                  
\end{tikzcd}
\end{center}
commutes, proving that $a_P \simeq T^{2n} (\tilde{a}_P)$. 

We now argue that $A_P \simeq a_P$. Observe first that $\underline{\End}_\cC(\cO_{D}) = \cO_{D} = \un_\cC$, with no extra identifications. Now, for a gluing pattern $P$ of rank $n$, 
\begin{align*}
\underline{\End}(\cO_{\Sigma_P}) & \simeq \underline{\End}\left( \cO_{\bigsqcup_{i=1}^n H_i} \right) \underset{\cC^{ \boxtimes 2n}}{\boxtimes} \cC  \\
& \simeq \left( \underline{\End}(\cO_H) \right)^{\boxtimes n} \underset{\cC^{\boxtimes 2n}}{\boxtimes} \cC \\
& \simeq \left( \underline{\End}_{\cC \boxtimes \cC} (\un_\cC) \right)^{\boxtimes n} \underset{\cC^{\boxtimes 2n}}{\boxtimes} \cC
\end{align*}
\end{proof}

\subsection{Closed surfaces}

In view of the previous subsection, to compute factorization homology of a closed surface we could first puncture it, compute the factorization of the new punctured surface and then use excision to obtain the computation for the original surface. The gluing happens then along a circle boundary component, i.e., along $S^1 \times I \simeq Ann$. We explain now how this computation is done. For this we need to present the annulus category $U(\cC):= \int_{Ann}\cC$ as a C*-tensor category. 

Denote by $\cC^{\sigma \op}$ the category $\cC$ equipped with the braiding $\sigma^{-1} = \sigma^*$. There are functors $\cC \to \cZ(\cC)$ and $\cC^{\sigma \op} \to \cZ(\cC)$ given by $X \mapsto (X, \sigma_{X,-})$ and $X \mapsto (X, \sigma_{-,X}^*)$, respectively. Combining them we obtain a functor $\cC \boxtimes \cC^{\sigma \op} \to \cZ(\cC)$ such that its composition with the forgetful functor $\cZ(\cC) \to \cC$ is the tensor product functor $T: \cC \boxtimes \cC \to \cC$. Actually, the data of $\cC \boxtimes \cC^{\sigma \op} \to \cZ(\cC)$ is equal to the data $(\cC \boxtimes \cC \to \cC, \tau)$, where $\tau_{X_1 \boxtimes X_2,Y}$ is the unitary natural isomorphism
$$(\sigma_{X_1,Y} \otimes \id_{X_2}) \circ (\id_{X_1} \otimes \sigma_{Y,X_2}^*) : T(X_1 \boxtimes X_2) \otimes Y = X_1 \otimes X_2 \otimes Y \to Y \otimes X_1 \otimes X_2 = Y \otimes T(X_1 \boxtimes X_2) \ . $$

Observe that $\tau$ has a unique extension to the algebraic ind-completion of $\cC \boxtimes \cC$.

The structure coming from the functor $\cC \boxtimes \cC^{\sigma \op} \to \cZ(\cC)$ is a particular example of a {\em quantum Manin pair} studied in \cite{SafronovMomentMaps}.
He proves, in the algebraic setting that $\underline{\End}(\un_\cC)$ is a commutative algebra in $\Vect(\cC \boxtimes \cC^{\sigma \op})$, and  since $T^R$, the right adjoint to the unitary tensor product functor $T: \cC \boxtimes \cC \to \cC$, is a $\cC \boxtimes \cC$-bimodule functor, the canonical equivalence $\cC \boxtimes \cC^{\sigma \op} \simeq \underline{\End}(\un_\cC) \- \Mod_{\cC \boxtimes \cC}$ is indeed an equivalence of monoidal categories (\cite{SafronovMomentMaps}, Proposition 2.16). Since $T^R$ is also an unitary functor, this is an equivalence of unitary tensor categories in our case. Using monadicity for base change (Lemma \ref{lemma:monadicityforbasechange}), we obtain

\begin{proposition}
There is an unitary monoidal equivalence
$$\cC \underset{\cC \boxtimes \cC^{\sigma \op}}{\boxtimes} \cC \simeq \cF \- \Mod_\cC \ ,$$
where in the right-hand side the monoidal structure comes from the half-braiding $\tau_{\underline{\End}_{\cC^{\boxtimes 2},-}}$ of $\cF$.
\end{proposition}

We shall write $\cC^{\op}$ to denote the category $\cC$ with the opposite tensor product. The canonical braided structure of $\cC^{\op}$ is given by
$$\sigma_{Y,X}^* : X \underset{\cC^\op}{\otimes} Y \to Y \underset{\cC^\op}{\otimes} X \ . $$
Then the identity functor on $\cC$ becomes a braided monoidal equivalence $\cC^{\sigma \op} \simeq \cC^{\op}$ when endowed with the monoidal structure
$$J_{X,Y} := \sigma_{Y,X}^* : X \underset{\cC^{\sigma \op}}{\otimes} Y \to X \underset{\cC^{\op}}{\otimes} Y \ .$$
Note that this monoidal structure is unitary, so that we have an unitary equivalence of braided C*-tensor categories.

\begin{corollary}
There is an equivalence of C*-tensor categories
$$\cC \underset{\cC \boxtimes \cC^{\op}}{\boxtimes} \cC \simeq \cF \- \Mod_\cC \ . $$
\end{corollary}

The factorization homology over the annulus is, in fact, a C*-tensor category, where the monoidal structure is the one induced by embedding two annuli inside a bigger one (Figure 3).

\begin{figure}[h]
\centering
\label{fig:tensorstructureoftheannulus}
\includegraphics[scale=.2]{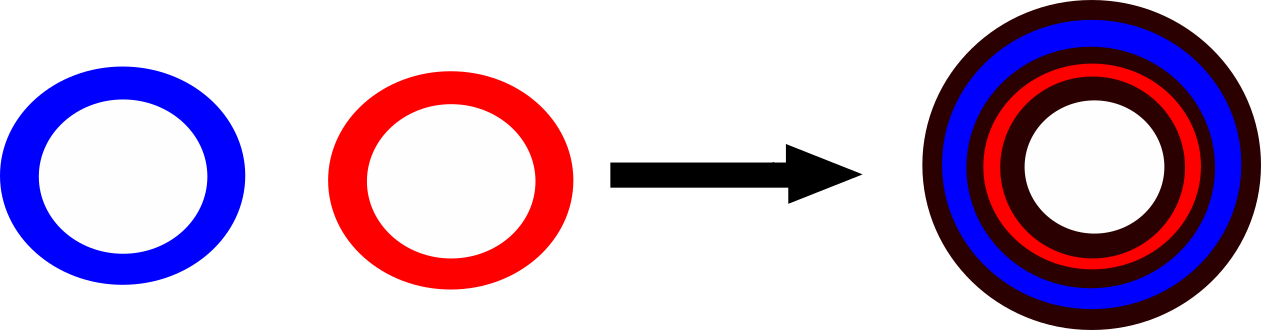} 
\caption{Tensor structure of the annulus category}
\end{figure}

\begin{theorem}
There is an unitary monoidal equivalence
$$\int_{Ann}\cC \simeq \cF \- \Mod_\cC \ . $$
\label{thm:facthomannulus}
\end{theorem}

\begin{proof}
Consider the following decomposition of the annulus (Figure 4):
\begin{figure}[h]
\centering
\label{fig:decompositionannulus}
\includegraphics[scale=0.2]{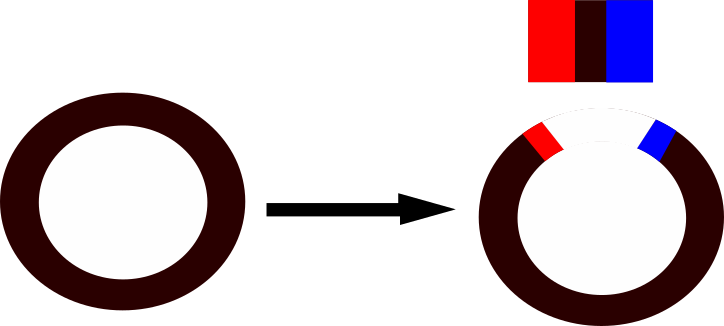} 
\caption{Decomposition of the annulus}
\end{figure}

Let $Y$ be the bottom surface at the right-hand-side of Figure 4. As an oriented surface, $Y$ is equivalent to the disk together with a prescribed embedding of $D \sqcup D$ into it. This corresponds via factorization homology to $\cC$ with a right action of $\cC \boxtimes \cC^{\op}$. Similarly, 

\begin{figure}[h]
\centering
\includegraphics[scale=.4]{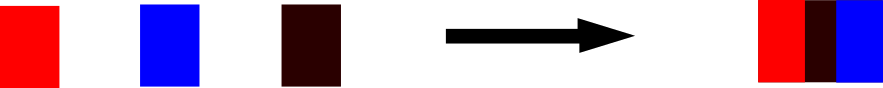} 
\label{fig:bimodulestructuredisk}
\caption{$\cC \boxtimes \cC^{\sigma \op}$-module structure of $\cC$}
\end{figure}
corresponds via factorization homology to a left action of $\cC \boxtimes \cC^{\op}$ on $\cC$. By excision,
$$\int_{Ann} \cC \simeq \cC \underset{\cC \boxtimes \cC^{ \op}}{\boxtimes} \cC \ ,$$
as C*-categories. The decomposition Figure 4 is compatible with the embedding in Figure 3: decomposing two annuli as in Figure 4 and then embedding as in Figure 3 is equivalent to embeddind first and then decompose. Therefore the above equivalence is also monoidal.

\end{proof}

Starting with a closed surface $\Sigma$, fix an embedding $D \to S$. Then $\Sigma^\circ := \Sigma \setminus D$ is a punctured surface, with respect to which factorization homology can be computed as in the previous subsection. The factorization homology of the punctured surface has, along with a $\cC$-module structure, an $U(\cC)$-module structure, induced by an embedding of the annulus along a collar of the boundary of $D \subset \Sigma$. 

To recover $\Sigma$ from $\Sigma^\circ$ is to perform the collar gluing
$$\Sigma \simeq \Sigma^0 \underset{Ann}{\sqcup} D \ , $$
and by excision
$$ \int_\Sigma \cC \simeq \int_{\Sigma^\circ} \cC \underset{U(\cC)}{\boxtimes} \cC \int_{D} \cC \simeq  \int_{\Sigma^\circ} \cC \underset{U(\cC)}{\boxtimes} \cC  \ . $$

\begin{theorem}
There is an equivalence of $U(\cC)$-module C*-categories:
$$\int_\Sigma \cC \simeq (A_{\Sigma^\circ} \- \un_\cC) \- \Mod_{\cF_\cC \- \Mod_\cC} \ . $$
\label{thm:facthomclosedsurfaces}
\end{theorem}
\begin{proof}
Follows from Corollary \ref{cor:balancedtensorproductovertheannuluscategory}.
\end{proof}

Using again Theorem \ref{prop:adjointofastarfunctor}, we can conclude
\begin{proposition}[\cite{BenZviBrochierJordan2}]
Let $\cO_\Sigma$ be the quantum structure sheaf of $\int_\Sigma \cC$. Then the endomorphism C*-algebra of $\cO_\Sigma$ can be described by a quantum Hamiltonian reduction
$$\End_{\int_\Sigma \cC}(\cO_\Sigma) \simeq \Hom_\cC(\un_\cC, A_{\Sigma^\circ} \underset{\cF_\cC}{\otimes} \un_\cC) \ , $$
where $A_{\Sigma^\circ} = \underline{\End}_\cC (\cO_{\Sigma^\circ})$. 
\label{prop:hamiltonianreduction}
\end{proposition}

Notice that, if $\int_\Sigma \cC$ had a $\cC$-module structure, Proposition \ref{prop:hamiltonianreduction} would mean that the internal endomorphism algebra of the quantum structure sheaf were $A_{\Sigma^\circ} \underset{\cF}{\otimes} \un_\cC$. That is not the case, since generally there is no way of embedding $\Sigma \sqcup D$ into $\Sigma$.

\section{Realization of cyclic representations}
\label{sec:realization}

Let $\cC$ be an UTC and let $A \in \Vect(\cC)$ be a C*-algebra object. We show here that, assuming $\cC$ acts on a $\II_1$ factor $N$, there is a way of realizing the cyclic representations of $A$ in terms of $N$. More precisely, we construct a functor from a category of cyclic representations of C*-algebra objects to a category of extensions of $N$. At the end of this section, assuming in addition that $\cC$ has an unitary braiding, we derive consequences of this realization functor in the context of factorization homology.

Suppose that $\cC$ acts on a $\II_1$-factor $N$. Denote by $H: \cC \to \Bim(N)$ the action functor, i.e., $H$ is a fully faithful unitary tensor functor from $\cC$ to the category of finite index Hilbert bimodules over $N$. Let $H^\circ: \cC \to \Vect$ be the space of bounded vectors functor:
$$H^\circ(X) := \{ \xi \in H(X) \ | \ \xi \ \text{is $N$-bounded} \} \ . $$

\begin{proposition}[\cite{JonesPenneys2}, Proposition 2.18]
$H^\circ$ is a W*-algebra object in $\Vect(\cC^{\op})$.
\end{proposition}

\begin{definition}
An algebra object $A \in \Vect(\cC)$ is called compact if $A \in \cC$, that is, if $A$ is an algebra in $\cC$. An algebra $A \in \Vect(\cC)$ is called connected iff $A(\un_\cC)$ is one dimensional.
\end{definition}

\begin{proposition}[\cite{JonesPenneys2}, Corollary 5.14]
If $A$ is a compact connected C*-algebra object in $\Vect(\cC)$, then $H^\circ(A)$ is a $\II_1$-factor having $N$ as a subfactor $N \overset{E_A}{\subset} H^\circ(A)$.
\end{proposition}

\begin{definition}
A subfactor $N \overset{E}{\subset} M$ with a faithful conditional expectation $E: M \to N$ is said to be irreducible iff $N' \cap M \simeq \C$, and it is said to be discrete iff, for $\phi:= \tau \circ E$, 
$$\tensor[_N]{\cL^2_\phi(M)}{_N} \simeq \bigoplus_{X \in \Bim_{sp}(N)} n_X X \ ,$$
with $n_X \in \N$ for all $X$ and $n_{\cL^2(N)} \neq 0$.
\end{definition}

The last proposition is a piece of an equivalence, the main result of \cite{JonesPenneys2}, between the category of connected C*-algebra objects in $\Vect(\cC)$, with ucp maps as morphisms, and the category $\DisInc_{\Irr}(N)$ of irreducible discrete subfactors $\{N \subset M\}$, where the morphisms between $N \subset M$ and $N \subset M'$ are ucp maps $M \to M'$ preserving the corresponding conditional expectations. This section provides a generalization of this result in one direction. Lets first recall how they construct the functors witnessing the equivalence.

Given a C*-algebra object $A \in \Vect(\cC)$, its algebraic realization relative to $H^\circ$ is given by 
$$|A| := |H^\circ \otimes A| = \bigoplus_{X \in \Irr(\cC)} H^\circ(X) \otimes A(X) \ , $$
which has a canonical $*$-algebra structure. Similarly, if $K \in \Hilb(\cC)$, then its Hilbert space realization relative to $H$ is given by
$$|K|:= |H \otimes K| = \bigoplus_{X \in \Irr(\cC)}^{\ell^2} H(X) \otimes K(X)  \ ,$$
where in the summand corresponding to $X \in \Irr(\cC)$ the inner product is normalized by a factor of $d_X^{-1}$.

\begin{theorem}[\cite{JonesPenneys2}, Theorem 4.11]
Let $A \in \Vect(\cC)$ be a C*-algebra object, and let $K \in \Hilb(\cC)$. If $\pi: A \to \bB(K)$ is a representation of $A$, then there is a unital $*$-algebra homomorphism $|\pi|: |A| \to \bB(|K|)$.
\label{thm:JonesPenneys2thm4.11}
\end{theorem}

Let $1_A$ be the unit of the ground C*-algebra $A(\un_\cC)$ of $A$. Note that there is always an inclusion
$$N \simeq H^\circ(\un_\cC) \otimes \C 1_{A} \subset |A|_{H^\circ} \ .$$
The homomorphism $|\pi|: |A| \to \bB(|K|)$ is induced by the map
$$H^\circ(X) \otimes A(X) \ni f \otimes g \mapsto f \otimes \pi_X(g) \in \bB(|K|) \ . $$
and there its restriction to $N$ is always faithful. Thus, if $M_{A,\pi}:= \left( |\pi|(|A|) \right)''$, there is a von Neumann algebraic inclusion $N \subset M_{A,\pi}$. Summarizing,

\begin{proposition}
There is a correspondence which associates to each representation $\pi$ of a C*-algebra object $A \in \Vect(\cC)$ a von Neumann algebraic inclusion $N \subset M_{A,\pi}$, where $M_{A,\pi}$ acts on the Hilbert space realization of the representation $\pi$.
\label{prop:realizationofrepresentations}
\end{proposition}

For a C*-algebra $B$, we shall write $\cS(B)$ for its state space. We state now the main theorem of this section.

\begin{theorem}
Let $\CAlg_\Omega(\cC)$ be the category whose objects are pairs $(A,\omega)$ consisting of a C*-algebra object $A \in \Vect(\cC)$ and a state $\omega \in \cS(A) := \cS(A(\un_\cC))$, and the morphisms between two objects $(A,\omega_A)$ and $(B,\omega_B)$ are ucp maps $\phi: A \to B$ mapping the state $\omega_A$ to the state $\omega_B$, i.e., such that $\omega_A = \omega_B \circ \phi_{\un_\cC}$. Let $\DisInc^\infty(N)$ be the category of subfactors $N \overset{E}{\subset} M$ such that $\tensor[_N]{\cL^2(M)}{_N}$ lies in the unitary ind-completion of $\Bim(N)$, morphisms being ucp maps preserving the embeddings of $N$. There exists a functor
$$\Xi: \CAlg_\Omega(\cC) \to \DisInc^\infty(N) \ . $$
If $\Xi(A,\omega) = N \overset{E}{\subset} M_{A,\omega}$, $E$ is faithful if and only if $\omega$ is faithful.
\label{thm:realizationofcyclicrepresentations}
\end{theorem}

We divide the proof into several lemmas and propositions. Let $(A,\omega) \in \CAlg_\Omega(\cC)$. Associated to $\omega$ there is Hilbert space object $\cL_{\omega}^2(A)$ and a representation $\pi_{\omega}: A \to \bB(\cL_{\omega}^2(A))$ (see \cite{JonesPenneys1}). Then Proposition \ref{prop:realizationofrepresentations} gives a von Neumann algebraic inclusion $N \subset M_{A,\pi_{\omega}} =: \Xi(A,\omega) \subset \bB(|\cL^2_\omega(A))$.

On $|A|$ there is a canonical state $\tau_\omega := |\tau| \otimes |\omega|$: given $a = \sum_{X \in \Irr(\cC)} a^X_{(1)} \otimes a^X_{(2)} \in |A|$, 
$$\tau_\omega(a) := \tau(a^{\un}_{(1)}) \cdot \omega(a^{\un}_{(2)}) \ . $$
There is also the GNS-map $A \to \cL^2_\omega(A)$, which induces a linear map $\Lambda_{\omega}: |A| \to |\cL^2_\omega(A)|$.

\begin{proposition}
For $a,b \in |A|$, $\tau_\omega(b^*a) = \left\langle \Lambda_\omega(b), \Lambda_\omega(a) \right\rangle$.
\label{prop:generalizationofCor5.4inJonesPenneys2}
\end{proposition}
\begin{proof}
Let $(\cdot)^{\natural}: A(X) \to A(\bar{X})$ be the involution on $A$. Then, denoting by $I(Z,X \otimes Y)$ a choice of maximal set of mutually orthogonal isometries $Z \to X \otimes Y$,
$$b^*a = \sum_{X,Y,Z \in \Irr(\cC)} \sum_{\gamma \in I(Z,X \otimes Y)} H^\circ (\gamma) \mu^{H^\circ}_{X,Y}(\overline{b^X_{(1)}} \otimes a^Y_{(1)} ) \otimes A(\gamma^*) \mu^A_{X,Y}((b^X_{(2)})^{\natural} \otimes a^Y_{(2)}) \ . $$
We assume that $I(\un,X \otimes Y)\simeq \delta_{X,\bar{Y}} \left( d_X^{-1/2} R_X \right)$. One then has
\begin{align*}
(b^*a)^\un & = \sum_{X \in \Irr(\cC)} d_X^{-1} H^\circ(R_X) \mu^{H^\circ}_{\bar{X},X}(\overline{b^X_{(1)}} \otimes aX_{^(1)})  \otimes A(R_X^*) \mu^A_{\bar{X},X}( (b^X_{(2)})^{\natural} \otimes a^X_{(2)}) \\
& = \sum_{X \in \Irr(\cC)} d_X^{-1} \left\langle b^X_{(1)}, a^X_{(1)} \right\rangle^{H^\circ}_N \otimes \left\langle b^X_{(2)},a^X_{(2)} \right\rangle_{A(X)} \ . 
\end{align*}
Therefore,
\begin{align*}
\tau_\omega(b^*a) = (\tau \otimes \omega)((b^*a)^\un) = \sum_{X \in \Irr(\cC)} \left\langle b^X_{(1)},a^X_{(1)} \right\rangle_{H(X)} \cdot \left\langle b^X_{(2)},a^X_{(2)} \right\rangle_{\cL_\omega^2 A(X)}  = \left\langle \Lambda_\omega(b), \Lambda_\omega(a) \right\rangle_{|H \otimes \cL_\omega^2 A |}\ . 
\end{align*}
\end{proof}

{\em Construction of the conditional expectation:} The Hilbert space $\cL_{\omega}^2(A)(\un_\cC)$ has a cyclic  vector $\Delta_\omega$ for the $A(\un_\cC)$-action. Consider the orthogonal projection
$$e: |H \otimes \cL_\omega^2(A)| \to \cL_\tau^2(N) \otimes \C \Delta_\omega \simeq \cL_\tau^2(N) $$
given by projection onto the direct summand indexed by $\un$. Under the identification $\cL^2_\tau(N) \simeq \cL^2_\tau(N) \otimes \C \Delta_\omega \subset |H \otimes \cL^2_\omega(A)|$, given $x \in M_{A,\pi_\omega}$, $e \circ (x|_{\cL_\tau^2(N)})$ is bounded and commutes with the right $N$-action, i.e., it lies in $N'' = N$. Therefore, 
$$\exists! \  E(x) \in N \ ; \ exe = E(x) e \ . $$

\begin{proposition}
Let $\cL_{\tau_\omega}|A|$ be the GNS-construction w.r.t. the state $\tau_\omega$ on $A$. There is an unitary isomorphism $U: \cL_{\tau_\omega}^2|A| \overset{\simeq}{\to} |\cL_{\omega}^2A|$.
\label{prop:comparisonofGNSconstructions}
\end{proposition}

\begin{proof}
Let $\tilde{\Lambda}:|A| \to \cL_{\tau_\omega}^2|A|$ and $\Lambda: |A| \to |\cL_\omega^2 A|$ be the GNS-maps. Define, for $a \in |A|$, $U(\tilde{\Lambda}(a)) := \Lambda(a)$. Then
\begin{align*}
\left\langle U \tilde{\Lambda}(a),U \tilde{\Lambda}(b) \right\rangle = \left\langle \Lambda(a), \Lambda(b) \right\rangle = \tau_\omega(a^*b) = \left\langle \tilde{\Lambda}(a),\tilde{\Lambda}(b) \right\rangle \ .
\end{align*}
The second equality follows from Proposition \ref{prop:generalizationofCor5.4inJonesPenneys2}. Since $\tilde{\Lambda}(|A|)$ and $\Lambda(|A|)$ are dense in $\cL_{\tau_\omega}|A|$ and $|\cL_\omega A |$, respectively, one concludes that $U$ extends to an unitary operator.
\end{proof}

Suppose now that $\omega$ is faithful. It follows then from Proposition \ref{prop:generalizationofCor5.4inJonesPenneys2} and Proposition \ref{prop:comparisonofGNSconstructions} that, for $a \in |A|$, $\tau_\omega(a^*a) = 0 \implies a = 0$, since the GNS-map $\Lambda_\omega: |A| \to |\cL_\omega^2(A)|$ is injective.

To prove faithfulness of $E$, we verify that
$$\tau \circ E = \langle (\cdot) \Omega \otimes \Delta_\omega, \Omega \otimes \Delta_\omega \rangle \ . $$
In other words, we verify that $\tau \circ E = \tau_\omega$.

Observe first that
$$ E( \cdot) e = e (\cdot) e \implies E(\cdot) = e ( \cdot) e^* \ , $$
where $e^*: \cL_\tau^2(N) \hookrightarrow | H \otimes \cL_\omega^2(A)|$ is the embedding $e^*(n \Omega) = n\Omega \otimes \Delta_\omega$. Therefore,
\begin{align*}
\tau(E(x))& = \langle E(x) \Omega, \Omega \rangle \\
& = \langle e x e^* \Omega, \Omega \rangle \\
&= \langle x e^* \Omega, e^* \Omega \rangle \\
&= \langle x( \Omega \otimes \Delta_\omega), \Omega \otimes \Delta_\omega \rangle \ ,
\end{align*}
for all $x \in M_{A,\omega}$. 

\begin{lemma}[\cite{JonesPenneys2}, Lemma 5.17]
Let $K_1,K_2 \in \Hilb(\cC)$, and suppose that $v: K_1 \to K_2$ is an isometry. Then $|v|: H(X) \otimes K_1(X) \ni \xi \otimes \eta \mapsto \xi \otimes (v \circ \eta )$ defines an isometry $|K_1| \to |K_2|$ such that $|v|^* = |v^*|$, where $|v^*|$ is defined similarly.
\end{lemma}

\begin{proof}
Given any Hilbert space object $K$, if $\xi \otimes \eta \in H(X) \otimes K(X)$ and $\xi' \otimes \eta' \in H(Y) \otimes K(Y)$, then $((\xi')^* \circ \xi) \circ ((\eta')^* \circ \eta)^{\op}) \in \cC(X)$, and $\left\langle \xi' \otimes \eta', \xi \otimes \eta \right\rangle = \tr_\cC((\eta')^* \circ \eta) \circ ((\xi')^* \circ \xi)^{\op})$. Thus, for $K_1,K_2$ and $v$ as in the Lemma, $\xi_i \otimes \eta_i \in H(X_i) \otimes K_1(X_i)$,
\begin{align*}
\left\| |v| \left( \sum_i \xi_i \otimes \eta_i \right) \right\|^2_{|K_2|} & = \left\| \sum_i \xi_i \otimes (v \circ \eta_i) \right\|^2_{|K_2|} \\
& = \sum_{i,j} \tr_\cC \left( (\eta_j^* v^* v \eta_i) \circ (\xi_j^* \xi_i)^{\op} \right) \\
&= \sum_{i,j} \left( (\eta_j^*  \eta_i) \circ (\xi_j^* \xi_i)^{\op} \right) \\
& = \left\| \sum_i \xi_i \otimes \eta_i \right\|^2_{|K_1|} \ .
\end{align*}
\end{proof}

\begin{proof}[Proof of Theorem \ref{thm:realizationofcyclicrepresentations}]
It remains to prove functoriality. Let $\theta: (A, \omega_A) \to (B,\omega_B)$ be a morphism in $\CAlg_\Omega(\cC)$. From $\omega_B$ we construct the Hilbert space object $\cL^2_{\omega_B}(B)$ and a $*$-representation $\pi_{\omega_B}: B \to \bB(\cL_{\omega_B}^2(B))$. Then the composite $\pi_{\omega_B} \circ \theta: A \to \bB(\cL_{\omega_B}^2(B))$ is a ucp map. By the Stinespring dilation Theorem in $\cC$ (\cite{JonesPenneys1}, Theorem. 4.18), there is a Hilbert space object $K \in \Hilb(\cC)$, a representation $\pi^A: A \to \bB(K)$ and an isometry $v: \cL_{\omega_B}^2(B) \to K$ such that $\theta^{\omega_B}:= \pi_{\omega_B} \circ \theta = \Ad_v \circ \pi^A$.

Define also $|\theta|: |A| \to |B|$ and $|\theta^{\omega_B}|: |A| \to \bB(|\cL_{\omega_B}^2|)$ by
$$f \otimes g \in H^\circ(X) \otimes A(X) \implies |\theta|(f \otimes g) = f \otimes \theta_X(g) \ \text{and} \ |\theta^{\omega_B}| (f \otimes g) = f \otimes \theta^{\omega_B}_X(g) \ . $$
By construction $|\theta^{\omega_B}| = |\pi_{\omega_B}| \circ |\theta|$. Then, with the isometry $|v|: |\cL_{\omega_B}^2| \to |K|$ of the previous lemma, $|\theta^{\omega_B}| = \Ad_{|v|} \circ |\pi^A|: |A| \to \bB(|\cL_{\omega_B}|)$.

Applying Corollary 4.13 in \cite{JonesPenneys2}, one concludes that $|\theta^{\omega_B}|$ satisfies Kadison's inequality: for all $x \in |A|$,
$$\left(|\theta|(x^*) \right) \left(|\theta|(x)\right) \leq |\theta|(x^*x) \ . $$

Denote by $\tau_{\omega_A}$ and $\tau_{\omega_B}$ the canonical states on $|A|$ and $|B|$ respectively. Being a vector state, $\tau_{\omega_B}$ extends  to a normal state on $\bB(|\cL^2_{\omega_B}|)$ (more precisely, first descends to the quotient $|B| / \ker(|\pi_{\omega_B}|)$ and then extends to $\bB(|\cL_{\omega_B}^2|)$). We claim that $|\theta|$ and $|\theta^{\omega_B}|$ map the state $\tau_{\omega_A}$ to $\tau_{\omega_B}$. Indeed, for any $a = \sum_{X} a^X_{(1)} \otimes a^X_{(2)} \in |A|$,
\begin{align*}
\tau_{\omega_B} \left(   |\theta| \left( \sum_{X} a^X_{(1)} \otimes a^X_{(2)} \right)  \right) &= (\tau \otimes \omega_B) \left( a^{\un_\cC}_{(1)} \otimes \theta_{\un_\cC}(a^{\un_\cC}_{(2)}) \right) \\
&= (\tau \otimes \omega_A) \left( a^{\un_\cC}_{(1)} \otimes a^{\un_\cC}_{(2)} \right) \ ,
\end{align*}
by the assumption that $\theta$ maps $\omega_A$ to $\omega_B$. The last expression is exactly $\tau_{\omega_A}(a)$, and one concludes then  that $\tau_{\omega_B} \circ |\theta| = \tau_{\omega_A}$. Similarly, one has $\tau_{\omega_B} \circ |\theta^{\omega_B}| = \tau_{\omega_A}$.

As in Corollary 5.23 in \cite{JonesPenneys2}, we can show now that $|\theta|$ induces a contraction $\cL^2|\theta|: |\cL_{\omega_A}^2| \to |\cL_{\omega_B}^2|$. Indeed, it follows from Proposition \ref{prop:comparisonofGNSconstructions} that these Hilbert space realizations have cyclic vectors $\Omega_A$ and $\Omega_B$, respectively. Define, for $x \in |A|$, 
$$\cL^2 |\theta| (x \Omega_A) := |\theta|(x) \Omega_B \ . $$ 
Then, for all $x \in \pi^A(|A|)$,
\begin{align*}
\| \cL^2 |\theta| (x \Omega_A) \|^2 &= \| |\theta|(x) \Omega_B \|^2 \\
& = \tau_{\omega_B} \left( \left(|\theta|(x^*) \right) \left(|\theta|(x) \right) \right)\\
& \leq \tau_{\omega_B} \left( |\theta| (x^*x) \right) \\
& = \tau_{\omega_A} (x^*x) \\
& = \| x \Omega_A \|^2 \ . 
\end{align*}

Let $a \in |A|$, $b,b' \in M_{B,\omega_B}'$. It follows that
\begin{align*}
\omega_{b\Omega_B,b'\Omega_B} (|\theta^{\omega_B}|(x)) & = \left\langle |\theta^{\omega_B}|(x) b\Omega_B, b'\Omega_B \right\rangle \\
& = \left\langle b |\theta^{\omega_B}|(x) \Omega_B,b'\Omega_B \right\rangle \\
&= \left\langle \cL^2 |\theta^{\omega_B}| (\pi^A(x) \Omega_A), b^*b' \Omega_B \right\rangle \\
&= \left\langle \pi^A(x) \Omega_A, \cL^2 |\theta^{\omega_B}|^* (b^* b' \Omega_B) \right\rangle
\end{align*}

Now, $(M_{B,\omega_B})'\Omega_B$ is dense in $|\cL_{\omega_B}| \simeq \cL_{\tau_{\omega_B}}|B|$.
It follows (\cite{JonesPenneys2}, Proposition 4.15) that $|\theta|$ and $|\theta^{\omega_B}|$ extend to normal ucp maps which can be put in a commutative diagram as follows:
\begin{center}
\begin{tikzcd}
{M_{A,\omega_A}} \arrow[rr, "|\theta|"] \arrow[rd, "|\theta^{\omega_B}|"'] & {} \arrow[loop, distance=2em, in=305, out=235] & {M_{B,\omega_B}} \arrow[ld, "|\pi_{\omega_B}|"] \\
                                                                           & \bB(\cL_{\omega_B}^2)                          &                                                
\end{tikzcd}
\end{center}

\end{proof}

\begin{remark}
Given a C*-algebra object $A$, its state space $\cS(A) = \cS(A(\un_\cC))$ is in bijection with the space of ucp multipliers $\cM_A \to \cC$, in the sense of \cite{HataishiYamashita}.
\end{remark}

\begin{example}
Consider the reflection equation algebra $\cF$ of $\cC$. It has a canonical state associated with the counit $\epsilon: \cF \to \cC$, which is in particular a ucp. multiplier. More concretely, the state we are considering is $\epsilon_{\un_\cC}: \cF(\un_\cC) \to \cC(\un_\cC) \simeq \C$. This state is in fact a character of $\cF(\un_\cC)$, i.e., a 1-dim. $*$-representation. Therefore, to every action $H: \cC \to \Bim(N)$ of $\cC$ on a $\II_1$-factor $N$, there is a canonical {\em reflection algebra extension} $N \overset{E}{\subset} M_{\cF,\epsilon_{\un_\cC}}$.
\end{example}

\subsection{Actions of mapping class groups on the realizations}

Let now $\cC$ be an unitarily braided unitary tensor category. Let $\Sigma$ be an oriented surface of genus $g$ and with $n+1$ boundary circle components. Each boundary component gives to $\int_\Sigma \cC$ the structure of a $\cC$-module C*-category. Fixing one of the boundary components $B$, there is a C*-algebra object $A_\Sigma$ in $\Vect(\cC)$ such that, w.r.t. the corresponding module structure,
$$\int_\Sigma \cC \simeq A_\Sigma \- \Mod_\cC \ . $$

Let $\Gamma^g_{n}$ be the group of orientation preserving diffeomorphisms of $\Sigma$ fixing $B$ pointwise, module isotopies. Automatically, we have the C*-algebraic version of Proposition 5.19 in \cite{BenZviBrochierJordan1}:

\begin{proposition}
There is a canonical action of $\Gamma_{n}^g$ on $A_\Sigma$ by $*$-algebra automorphisms.
\label{prop:actionofmappingclassgrouponalgebraobjects}
\end{proposition}

Assume now that $\cC$ acts on $\Bim(N)$ as in the previous section. The action $\Gamma_{n,1}^g \curvearrowright A_\Sigma$ induces an action $\Gamma_{n,1}^g \curvearrowright \cS(A_M)$. Applying Theorem \ref{thm:realizationofcyclicrepresentations} gives then the following corollary:

\begin{corollary}
Given $\omega \in \cS(A_\Sigma)^{\Gamma_{n}^g}$, the action of $\Gamma_{n}^g$ on $A_\Sigma$ induces an action $\Gamma_{n}^g \curvearrowright M_{A_\Sigma,\omega}$ by normal $*$-homomorphisms such that $N \subset (M_{A_\sigma,\omega})^{\Gamma_{n}^g}$.
\label{cor:actionofmappingclassgrouponextensions}
\end{corollary}

The following theorem points out to the possibility of existence of canonical invariant states for the actions of mapping class groups.

\begin{theorem}
Let $\cC$ be an unitary tensor category equipped with an unitary braiding. Let $\cF$ be the corresponding reflection equation algebra. Then the counit $\epsilon: \cF \to \un$ induces a $\Gamma_{2}^0$-invariant state on $\cF$.
\label{thm:invarianceofthecounitundermappingclassgroup}
\end{theorem}

\begin{proof}
The mapping class group $\Gamma_{2}^0$ of the annulus is isomorphic to $\Z$; a generator is given by $\delta$
\begin{figure}[h]
  \centering
  \includegraphics[width=.25\linewidth]{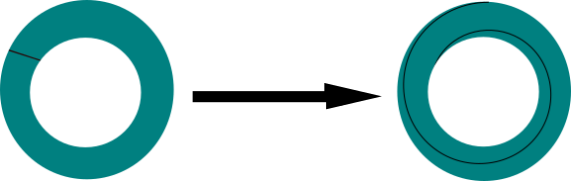}
\caption{Generator of $\Gamma_{2}^0$}
\label{fig:generatorofGamma0,2}
\end{figure}. 
The counit $\epsilon$ corresponds under factorization homology to the embedding of the annulus into the disk. This embedding induces a group homomorphism $\Gamma_{2}^0 \to \Gamma_{1}^0$ from the mapping class group of the annulus to the mapping class group of the disk. The latter is however trivial. Hence there is an isotopy of the embedding $Ann \into D$ completing the diagram in Figure \ref{fig:invariancecounti} to a homotopy commutative diagram, showing that $\epsilon$ must be $\Gamma_{2}^0$-invariant.
\begin{figure}[h]
\centering
\includegraphics[width=.25\linewidth]{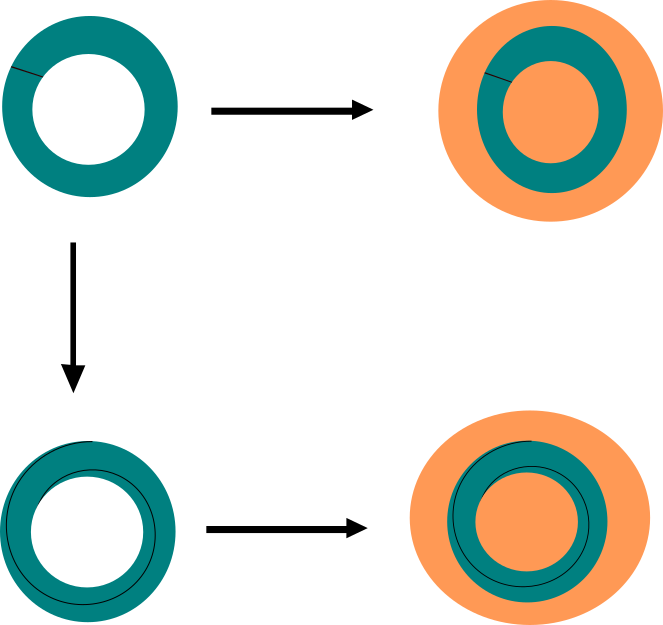}
\caption{Proof of invariance of the counit}
\label{fig:invariancecounti}
\end{figure}. 
\end{proof}

A given unitary tensor category $\cC$ may have several unitary braidings.

\begin{corollary}
Let $N$ be a $\II_1$-factor and let $\cC \subset \Bim(N)$ be a unitary tensor subcategory of finite index bimodules over $N$. To every choice of a unitary braiding $\beta$ on $\cC$, there is an associated extension
$$N \overset{E}{\subset} \Xi(\cF_\beta,\epsilon) \ , $$
where $\cF_\beta$ is the reflection equation algebra of $\cC$ associated to the unitary braiding $\beta$. Moreover, the conditional expectation $E$ is faithful.
\label{cor:extensionsbyreflectioneqalgebra}
\end{corollary}
\begin{proof}
The existence of the extension follows from Proposition \ref{prop:actionofmappingclassgrouponalgebraobjects}, Corollary \ref{cor:actionofmappingclassgrouponextensions}, Theorem \ref{thm:invarianceofthecounitundermappingclassgroup} and Corollary \ref{cor:extensionsbyreflectioneqalgebra}. Regarding faithfulness, recall that
$$\cF_\beta(\un) \simeq \bigoplus_{i \in \Irr(\cC)}^{c_0} \C R_i \ . $$
Notice that, since $\epsilon$ is a character of $\cF_\beta$, the induced state on $\cF_\beta(\un)$ is also a character, which we still denote by $\epsilon$. Explicitly,
$$\epsilon = \sum_{i \in \Irr(\cC))} d_i^{-1} R_i^* \ .  $$
Hence,
$$\epsilon \left( \left( \sum_{i \in \Irr(\cC)} \lambda_i R_i \right)^* \left( \sum_{i \in \Irr(\cC)} \lambda_i R_i \right) \right) = \left( \sum_{i \in \Irr(\cC)} \lambda_i  \right)^* \left( \sum_{i \in \Irr(\cC)} \lambda_i \right) \geq 0 \ , $$
the equality holding only if all coefficients $\lambda_i$ are zero. By Theorem \ref{thm:realizationofcyclicrepresentations}, the conditional expectation associated to $\epsilon$ via the realization functor is faithful.
\end{proof}

Let $G$ be a second countable compact group. In \cite{FalguieresVaes}, it was proven that there exists a $\II_1$-factor $M$ and a minimal action $G \curvearrowright M$ such that the category $\Bim(M^G)$ of finite index bimodules of the fixed point $\II_1$-factor $M^G$ is unitarily monoidally equivalent to $\Rep G$. It is known that $\Rep G$ has unitary braidings coming from non-degenerate abelian subgroups of $G$. Through the unitary monoidal equivalence $\Bim(M^G) \simeq \Rep G$, we can transport unitary braidings on the latter to unitary braidings on the former.

\section{Hamiltonian reduction in the fusion case}
\label{sec:quantumhamiltoninareduction}
We now interpret quantum Hamiltonian reduction in terms of von Neumann algebras. Let $\cC$ be an unitary braided fusion full subcategory of $\Bim_0(N)$, $N$ being either a type $\II_1$ or a type $\III$ factor. Denote by $H: \cC \into \Bim(N)$ the inclusion. Given a C*-algebra $A \in \cC$, the space of bounded vectors $H^\circ(A)$ in the bimodule $H(A)$ has a canonical structure of a von Neumann algebra: for the type $\II_1$ case see the previous session  for further discussion and references, and for the type $\III$ case note that, for every $X \in \cC$, $H(X) \simeq \cL^2(N)$ by means of an endomorphism $\rho_X$ of $N$, $\cL^2(N)$ being constructed using a faithful normal semi-finite weight. Therefore, $H^\circ(X) \simeq N$ as Banach spaces. We shall denote $H^\circ(X)$ identified with $N$ as $N_X$. Hence, there is a linear isomorphism
$$H^\circ(A) \simeq \bigoplus_{X \in \Irr(\cC)} H^\circ(X) \otimes \cC(X,A) \simeq \bigoplus_{X \in \Irr(\cC)} N _X\otimes \cC(X,A) \ , $$
each $\cC(X,A)$ being finite dimensional and $\Irr(\cC)$ being a finite set. Define a map $N_X \otimes N_Y \to N_{X \otimes Y}$ by the formula
$$n \otimes m \mapsto n \rho_X(m) \ . $$
Then the algebra structure of $H^\circ(A)$ is determined by the maps 
\begin{align*}
N_X \otimes \cC(X,A) \otimes N_Y \otimes \cC(Y,A) &\to N_{X \otimes Y} \otimes \cC(X \otimes A) \\
n \otimes f \otimes m \otimes g &\mapsto n \rho_X(m) \otimes m_A \circ (f \otimes g) \ ,
\end{align*}
together with decomposition of tensor products into irreducible objects. Here, $m_A: A \otimes A \to A$ is the multiplication morphism of $A$.

\begin{remark}
The passage from a C*-algebra $A \in \cC$ to the inclusion $N \subset H^\circ(A)$ is essentially the Q-system reconstruction theorem, phrased in terms of bimodules instead of sectors (\cite{BischoffKawahigashiLongoRehrenbook}).
\end{remark}

We shall however use the description
$$H^\circ(A) \simeq H^\circ(X) \otimes \cC(X,A) \ ,$$
under which the multiplication becomes
$$(\xi \otimes f) \otimes (\eta \otimes g) \mapsto
\mu^H(\xi \otimes \eta) \otimes m_A (f \otimes g) \ $$
for $\xi \otimes f \in H^\circ(A) \otimes \cC(X,A)$ and $\eta \otimes g \in H^\circ(Y) \otimes \cC(Y,A)$, $\mu^H$ being the monoidal structure of $H: \cC \into \Bim_0(N)$. Writing $H(\bar{X}) = \{ \bar{\xi} \ | \ \xi \in H(X) \}$, the $*$-structure of $H^\circ(A)$ is given by
$$ \xi \otimes f \mapsto \bar{\xi} \otimes f^\sharp \ ,$$
for $\xi \otimes f \in \cC(X,A)$, $( \cdot)^\sharp: \cC(X,A) \to \cC(\bar{X},A)$ being the *-structure of $A$.

\begin{proposition}
Let $A$ be a C*-algebra object in $\cC$, and suppose that  $\epsilon: A \to \un_\cC$ is a $*$-homomorphism such that $H^\circ(\epsilon): H^\circ(A) \to N$ is surjective. Suppose that $X$ is a right $A$-module in $\cC$. Given another object $Y$, consider the left $A$-action on $Y$ induced by $\epsilon$: $A \otimes Y \mapsto \un_\cC \otimes Y \simeq Y$. Then
$$H(X \underset{A}{\otimes} Y) \simeq H(X) \underset{H^\circ(A)}{\otimes}H(Y) $$
as  $N\-N$-bimodules. 
\label{prop:relativetensorproductHilbertbimodules}
\end{proposition}

\begin{proof}
The object $X \underset{A}{\otimes} Y$ is defined as the colimit of the diagram
\begin{center}
\begin{tikzcd}
X \otimes A \otimes Y \arrow[r, shift right] \arrow[r, shift left] & X \otimes Y \ ,
\end{tikzcd}
\end{center}
where one of the arrows is the action of $A$ on $X$ and the other the action of $A$ on $Y$. Since $H$ is full, $H(X \underset{A}{\otimes} Y )$ is the colimit in $\Bim(N)$ of the diagram
\begin{center}
\begin{tikzcd}
H(X) \otimes H(A) \otimes H(Y) \arrow[r, shift right] \arrow[r, shift left] & H(X) \otimes H(Y) \ .
\end{tikzcd}
\end{center}
The action of $H^\circ(A)$ on $H(Y)$ is defined through the $*$-homomorphism $H(\epsilon): H^\circ(Y) \to N$, which is surjective. Therefore, the space of $H^\circ(A)$-bounded vectors in $H(Y)$ coincides with the space of $N$-bounded vectors. Thus, there is an idempotent $N$-$N$-bimodule intertwiner $P: H(X) \underset{N}{\otimes} H(Y) \to H(X) \underset{H^\circ(A)}{\otimes} H(Y)$. If $\cK \in \Bim(N)$ is such that the diagram
\begin{center}
\begin{tikzcd}
H(X) \underset{N}{\otimes} H(A) \underset{N}{\otimes} H(Y) \arrow[rr] \arrow[rr, shift left] \arrow[rdd, bend right, shift left] \arrow[rdd, bend right] & {} \arrow[r] & H(X) \underset{N}{\otimes} H(Y) \arrow[ldd, "T", bend left] \\
                                                                                                                                                                     &              &                                                                       \\
                                                                                                                                                                     & \cK          &                                                                      
\end{tikzcd}
\end{center}
commutes, define $T'(\xi \underset{H^\circ(A)}{\otimes} \eta) := T(\xi \underset{N}{\otimes} \eta)$. Then $T'$ has a unique extension to a $N$-$N$-bimodule intertwiner $H(X)\underset{H^\circ(A)}{\otimes} H(Y) \to \cK$ such that
\begin{center}
\begin{tikzcd}
H(X) \underset{N}{\otimes} H(A) \underset{N}{\otimes}H(Y) \arrow[rr, shift left] \arrow[rr, shift right] \arrow[rd] \arrow[rd, shift left] \arrow[rdd, bend right] &                                                                  & H(X) \underset{N}{\otimes} H(Y) \arrow[ld, "P"] \arrow[ldd, "T", bend left] \\
                                                                                                                                                                   & H(X) \underset{H^\circ(A)}{\otimes}H(Y) \arrow[d, "T'"', dashed] &                                                                             \\
                                                                                                                                                                   & \cK                                                              &                                                                            
\end{tikzcd}
\end{center}
commutes.
\end{proof}

Recall the reflection equation algebra $\cF$. Our goal now is to understand better the realization $H^\circ(\epsilon)$ of the counit $\epsilon: \cF \to \un$ in order to prove a statement about quantum Hamiltonian reduction.

Recall that the counit $\epsilon$ is explicitly given by $\sum_{X \in \Irr(\cC)} d_X^{-1} R_X^*: \bigoplus_{X} \bar{X} \otimes X \to \un_\cC$.

Given $a \in H^\circ(\cF)$, write
$$a = \sum_{X \in \Irr(\cC)} a^X_{(1)} \otimes a^X_{(2)} \  . $$

Note that the vector space $\cF(\un_\cC) = \cC(\un_\cC, \cF)$ is the linear span of $\{ R_X \ | \ X \in \Irr(\cC)\}$. Thus, for $a \in H^\circ(\cF)$ as above,
$$a^\un_{(1)} \otimes a^{\un}_{(2)} = \sum_{Y \in \Irr(\cC)} a^{\un}_Y \otimes R_Y \ , \ a^{\un}_Y \in N \ . $$

Therefore
$$H^\circ(\epsilon) (a) = H^\circ(\epsilon) \left( \sum_X a^X_{(1)} \otimes a^X_{(2)} \right) = \sum_X a^{\un}_X \ . $$
\begin{proposition}
Let $\epsilon: \cF \to \un_\cC$ be the counit homomorphism. Then $H^\circ: H^\circ(\cF) \to N$ is a normal $*$-homomorphism.
\label{prop:normalityofthecounit}
\end{proposition}

\begin{proof}
We have to show that, with the notation used above, if $a_i \to a$ weakly in the unit ball of $H^\circ(\cF)$, then $\sum_X (a_i)^{\un}_X \to \sum_X a^{\un}_X$ in $N$.

{\em Step 1:} A net $(a_i)_i \subset H^\circ(\cF)$, $a_i = \sum_X a^X_{i,(1)} \otimes a^X_{i,(2)}$ converges weakly to  $a = \sum_X a^X_{(1)} \otimes a^X_{(2)}$ if and only if
$$\sum_{X,Y \in \Irr(\cC)} \sum_{Z \leq X \otimes Y} \left\langle \mu^H((a^X_{i,(1)} \otimes \xi^Y_{(1)}) \otimes \mu^\cF(a^X_{i,(2)} \otimes \xi^Y_{(2)}) ) , \eta^Z_{(1)} \otimes \eta^Z_{(2)} \right\rangle $$
converges to
$$\sum_{X,Y \in \Irr(\cC)} \sum_{Z \leq X \otimes Y} \left\langle \mu^H((a^X_{(1)} \otimes \xi^Y_{(1)}) \otimes \mu^\cF(a^X_{(2)} \otimes \xi^Y_{(2)}) ) , \eta^Z_{(1)} \otimes \eta^Z_{(2)} \right\rangle $$

{\em Step 2:} In Step 1, taking $Y = Z = \un_\cC$, the only non trivial summand is $X = \un_\cC$, and then it reads
$$ \left\langle \mu^H((a^{\un_\cC}_{i,(1)} \otimes \xi^{\un_\cC}_{(1)}) \otimes \mu^\cF(a^{\un_\cC}_{i,(2)} \otimes \xi^{\un_\cC}_{(2)}) ) , \eta^{\un_\cC}_{(1)} \otimes \eta^{\un_\cC}_{(2)} \right\rangle \to \left\langle \mu^H((a^{\un_\cC}_{(1)} \otimes \xi^{\un_\cC}_{(1)}) \otimes \mu^\cF(a^{\un_\cC}_{(2)} \otimes \xi^{\un_\cC}_{(2)}) ) , \eta^{\un_\cC}_{(1)} \otimes \eta^{\un_\cC}_{(2)} \right\rangle \ . $$

{\em Step 3:} One has
\begin{align*}
\xi^{\un}_{(1)} \otimes \xi^\un_{(2)} &= \sum_Y \xi^{\un}_Y \otimes R_Y  \\
\eta^{\un}_{(1)} \otimes \eta^{\un}_{(2)} &= \sum_Z \eta^{\un}_Z \otimes R_Z \ .
\end{align*}
Thus the convergence in Step 2 becomes
$$\sum_{X,Y} \sum_{Z \leq X \otimes Y} \left\langle \mu^H \left( (a_i)^{\un}_X \otimes \xi^{\un}_Y \right), \eta^{\un}_Z \right\rangle \to \sum_{X,Y} \sum_{Z \leq X \otimes Y} \left\langle \mu^H \left(a^{\un}_X \otimes \xi^{\un}_Y \right), \eta^{\un}_Z \right\rangle \ , $$
i.e.,
$$\sum_{X,Y} \sum_{Z \leq X \otimes Y} \left\langle (a_i)^{\un}_X \triangleright \xi^{\un}_Y , \eta^{\un}_Z \right\rangle \to \sum_{X,Y} \sum_{Z \leq X \otimes Y} \left\langle a^{\un}_X \triangleright \xi^{\un}_Y , \eta^{\un}_Z \right\rangle \ , $$

{\em Step 4:} There are inclusions $\cL^2(N) \into \cL^2(N) \otimes \cF^(\un) \into H(\cF)$, where a vector $\xi \in \cL^2(N)$ corresponds to the vector
$$\sum_X \xi^X_{(1)} \otimes \xi^X_{(2)} = \xi^{\un}_{(1)} \otimes \xi^{\un}_{(2)} = \sum_X \xi^{\un}_X \otimes R_X \ , $$
and where in the last expression the only non-zero factor is when $X = \un$, in which case $\xi^{\un}_{\un}  =\xi$. Then Step 3 reduces to
$$\sum_{X}  \left\langle (a_i)^{\un}_X \triangleright \xi , \eta \right\rangle \to \sum_{X} \left\langle a^{\un}_X \triangleright \xi , \eta \right\rangle \ .$$

\end{proof}

Given a closed surface $\Sigma$, we were able to describe  factorization homology over $\Sigma$ in terms of $A_{\Sigma^\circ}$, the C*-algebra in $\Vect(\cC)$ representing factorization homology of $\Sigma^\circ : \Sigma \setminus D$, and the reflection equation algebra $\cF$. Observe that, since here  $\cC$ is a fusion category, both $A_{\Sigma^{\circ}}$ and $\cF$ are compact, i.e., they are C*-algebras in $\cC$. Using $H$ we can also realize the quantum Hamiltonian reduction of Proposition \ref{prop:hamiltonianreduction} using Propostion \ref{prop:relativetensorproductHilbertbimodules}:

\begin{theorem}
There is an isomorphism
$$\End_{\int_\Sigma \cC} (\cO_S) \simeq \Hom_{\Bim(N)} (\cL^2(N), H(A_{\Sigma^0}) \underset{M_{\cF}}{\otimes} \cL^2(N)) \  $$
between the endomorphisms of the quantum structure sheaf on $\Sigma$ and the central vectors in $H(A_{\Sigma^\circ}) \underset{H^\circ(\cF)}{\otimes} \cL^2(N)$.
\label{thm:vonNeumannalgebraicquantumHamiltonianreduction}
\end{theorem}

\begin{proof}
By Proposition \ref{prop:hamiltonianreduction}, and since $H$ is a fully faithful $*$-functor, there is an isomorphism 
$$\End_{\int_\Sigma \cC}(\cO_\Sigma) \simeq \Hom_{\Bim(N)}(\cL^2(N), H(A_{\Sigma^0} \underset{M_{\cF}}{\otimes} \cL^2(N))) \ . $$
The result follows now from Proposition \ref{prop:relativetensorproductHilbertbimodules}.
\end{proof}

\appendix

\section{Yetter-Drinfeld C$^*$-algebras and centrally cyclic bimodule C$*$-categories}
\label{app:YetterDrinfeldC*algebras}

This Appendix gives a characterization of certain bimodule C$^*$-categories over a given UTC $\cC$ in terms of internal C$^*$-algebra objects. Among the examples are C$^*$-tensor categories $\cB$ for which the bimodule structure is induced by a unitary tensor functor $F:\cC \to \cB$.

\begin{definition}
Let $\cC$ be an UTC. A centrally pointed $\cC$-bimodule C$^*$-category $(\cM,m,\alpha)$ is a $\cC$-bimodule C$^*$-category $\cM$ together with a distinguished object $m \in \cM$ and an unitary natural isomorphism
$$\alpha = \{\alpha_U : m \triangleleft U \simeq U \triangleright m\}_{U \in \cC} $$
which is coherent wrt. the tensor structure of $\cC$. If $(\cM,m,\alpha)$ is a centrally pointed $\cC$-bimodule C$^*$-category such that $m$ is cyclic, we say that $(\cM,m,\alpha)$ is a centrally cyclic $\cC$-bimodule C$^*$-category. 
\label{def:centrallycyclicbimodules}
\end{definition}

When $\cC = \Rep \bG$ is the representation category of a compact quantum group $\bG$, it is shown in \cite{HataishiYamashita} that the category of centrally pointed bimodule C$^*$-categories is equivalent to Yetter-Drinfeld $\bG$-C$^*$-algebras, i.e., the category of continuous actions of the Drinfeld double $\cD\bG$ of $\bG$ on unital C$^*$-algebras. This will justify Definition \ref{def:Yetter-DrinfeldC*algebra} below.

Up to unitary equivalence of $\cC$-module C$^*$-categories, we can assume that $(\cM,m)$ is, as a cyclic right $\cC$-module C$^*$-category, of the form $(\cM_A,A)$, where $A = (A,\mu)$ is a C$^*$-algebra object in $\Vect(\cC)$ and $\cM_A$ is the Karoubi completion of the full subcategory of free left $A$-modules of the form $A \otimes U$, with $U \in \cC$.

\begin{lemma}
Let $(\cM_A,A)$ be a cyclic right $\cC$-module C$^*$-category as above and let $\cC \times \cM_A \to \cM_A$ be a left $\cC$-module C$^*$-structure on $\cM_A$. Suppose that there is an unitary natural isomorphism 
$$\alpha: \cC \triangleright A \simeq A \triangleleft \cC$$
which is coherent wrt. the tensor structure of $\cC$. Then there are unitary isomorphisms
$$U \triangleright A \simeq U \otimes A \ , $$
natural in $U$.
\end{lemma}
\begin{proof}
Since the right $\cC$-module structure is given by $A \triangleleft U = A \otimes U$, we have
\begin{align*}
\cM_A(U \triangleright A,A)  \overset{\alpha_U \circ (-)}{\simeq} \cM_A(A \otimes U, A) \simeq \Vect(\cC)(U,A) \ .
\end{align*}
Hence, the internal endomorphism of $A$ wrt. the left $\cC$ action is $A$ itself.
\end{proof}

Since every object of $\cM_A$ is a subobject of one of the form $A \otimes U$, the Lemma above says that the left $\cC$-action on $\cM_A$ is unitarily equivalent to $V \triangleright (A \otimes U) = V \otimes A \otimes U$. Under this identification, the unitary isomorphism $\alpha_U:A \otimes U \simeq U \triangleright A $ becomes an isomorphism $\sigma_U: A \otimes U \simeq U \otimes A$, coherent wrt. the the tensor structure of $\cC$. This means exactly that $\sigma = \{\sigma_U\}_{U \in \cC}$ is a half-braiding on the object $A \in \Vect(\cC)$. Moreover, each component $\sigma_U$ is a morphism of left $A$-modules (i.e., a morphism in $\cM_A$), where the object $U \otimes A$ has the left $A$-module structure given by
$$A \otimes U \otimes A \overset{\id_A \otimes \sigma_U^{-1}}{\simeq} A \otimes A \otimes U \overset{\mu \otimes \id_U}{\longrightarrow} A \otimes U \overset{\sigma_U}{\simeq} U \otimes A \ . $$
Equivalently, 
$$\sigma_U (\mu \otimes \id_U) = (\id_U \otimes \mu)(\sigma_U \otimes \id_A)(\id_A \otimes \sigma_U) \ . $$
This equation characterizes the multiplication $\mu: A \otimes A \to A$ as a morphism $(A,\sigma) \otimes (A,\sigma) \to (A,\sigma)$ of half-braidings in $\Vect(\cC)$.

\begin{remark}
In the above, objects of the form $U \otimes A$, endowed with the left $A$-module structure induced by the half-braiding $\sigma$, are strictly speaking not objects in the category $\cM_A$, but isomorphic as left $A$-modules to objects of that category. We shall however abuse notation and write for example $\cM_A(U\otimes A, A \otimes V)$. This will not cause any trouble since we are always considering full subcategories of the category of left $A$-modules.
\end{remark}

Observe that, starting with a centrally cyclic $\cC$-bimodule C$^*$-category of the form $(\cM_A,A,\alpha)$ and considering the induced half-braiding $(A,\sigma)$, the unitarity of $\alpha_U \in \cM_A(A \otimes U,U \triangleright A)$ corresponds to the unitarity of $\sigma_U$ as an element of $\cM_A(A \otimes U, U \otimes A)$.

\begin{definition}
Let $\cC$ be an UTC. A C$^*$-algebra object in $\cZ(\Vect(\cC))$ is an algebra $(A,\mu,\sigma) \in \cZ(\Vect(\cC))$ together with a $*$-algebra structure $j$ such that, wrt. the corresponding dagger structure on $(\cM_A,A)$, the components $\sigma_U \in \cM_A(A \otimes U, U \otimes A)$ of $\sigma$ are unitary and such that the $*$-algebra $(A,\mu,j)$ in $\Vect(\cC)$ is a C$^*$-algebra. A Yetter-Drinfeld C$^*$-algebra in $\Vect(\cC)$ is a C$^*$-algebra object in $\cZ(\Vect(\cC))$.
\label{def:Yetter-DrinfeldC*algebra}
\end{definition}

The discussion preceding Definition \ref{def:Yetter-DrinfeldC*algebra} proves

\begin{theorem}
Let $\cC$ be an UTC. Then the category of centrally cyclic $\cC$-bimodule C$^*$categories is equivalent to the category of Yetter-Drinfeld C$^*$-algebras in $\cC$. 
\end{theorem} 

\begin{definition}
Let $\cC$ be an UTC and $A=(A,\mu,j)$ be a C$^*$-algebra object in $\Vect(\cC)$. A unitary half-braiding $\sigma$ on $A$ is a half-braiding on $A$ such that $(A,\mu,\sigma,j)$ is a Yetter-Drinfeld C$^*$-algebra in $\Vect(\cC)$.
\label{def:unitaryhalfbraiding}
\end{definition}

\section{Maximal braided tensor product}
\label{app:Braidedtensorproduct}

A prominent construction in factorization homology is the braided tensor product of module categories. In this section we provide a concrete description of maximal braided tensor products between cyclic $\cC$-module C*-categories. The discussion however makes sense in a more general context; we don't need to assume that $\cC$ has a braiding, and the braided tensor product will still make sense as long as the algebras representing the module categories possess an additional structure, that of Yetter-Drinfled algebras. We shall discuss such structures shortly.

Let $\cC$ be an unitary tensor category, $\cM_A$ a cyclic $\cC$-bimodule C*-category and $\cM_B$ be cyclic left $\cC$-module $C^*$-category, corresponding to the C*-algebras $A$ and $B$, respectively. Write $\cM_A \underset{\cC}{\boxtimes} \cM_B$ for their maximal, relative tensor product (see \cite{AntounVoigt}, Section 5). To define a left action of $\cC$ on $\cM_A \underset{\cC}{\boxtimes} \cM_B$ is the same as defining a tensor functor $\cC \to \End( \cM_A \underset{\cC}{\boxtimes} \cM_B)$. This is the same as defining a $*$-functor $\cC \to \Bal_\cC(\cM_A \underset{\max}{\boxtimes} \cM_B , \cM_A \underset{\cC}{\boxtimes} \cM_B)$.

Denote by $\gamma: \cM_A \underset{\max}{\boxtimes} \cM_B \to \cM_A \underset{\cC}{\boxtimes} \cM_B$ the canonical $\cC$-balanced $*$-functor. For $U \in \cC$, let
$$U \blacktriangleright (-) := \gamma \circ (U \triangleright (-) \boxtimes \id_{\cM_B} ) : \cM_A \underset{\max}{\boxtimes} \cM_B \to \cM_A \underset{\cC}{\boxtimes} \cM_B \ . $$
If $\beta = \{ \beta_{X_A,Y_B;V}: \gamma(X_A \triangleleft V \boxtimes Y_B) \simeq \gamma(X_A \boxtimes V \triangleright Y_B) \}_{X_A,Y_B,V}$ is the  balancing structure of $\gamma$, then 
$$ \beta^U := \left\lbrace \beta_{U \triangleright X_A,Y_B;V}: U \blacktriangleright (X_A \triangleleft V \boxtimes Y_B) \simeq U \blacktriangleright (X_A \boxtimes V \triangleright Y_B) \right\rbrace_{X_A,Y_B,V} $$
is a balancing structure on $U \blacktriangleright (-)$. If $T \in \Hom_\cC(U,U')$, a natural transformation from $U \blacktriangleright (-)$ to $U' \blacktriangleright (-)$ is given by $\gamma( T \triangleright (-) \boxtimes \id_{\cM_B})$. We obtain thus an action $\cC \curvearrowright \cM_A \underset{\cC}{\boxtimes} \cM_B$, and the same reasoning, exchanging left and right actions, gives a left action $\cM_A \underset{\cC}{\boxtimes} \cM_B \curvearrowleft \cC$. We will denoted these action by $\triangleright$ and $\triangleleft$, respectively.

We claim that, if $A$ and $B$ are Yetter-Drinfeld C*-algebras in $\Vect(\cC)$, the C*-algebra object representing $\cM_A \underset{\cC}{\boxtimes} \cM_B$ is also a Yetter-Drinfeld C*-algebra.

Suppose now that $\cM$ and $\cN$ are cyclic bimodule C$^*$-categories, with central generators $m$ and $n$, respectively. That is to say, $\underline{\End}(m)$ and $\underline{\End}(n)$ are Yetter-Drinfeld C*-algebras in $\Vect(\cC)$. As a direct consequence of the definitions and the above constructions, we see that there is a canonical unitary natural isomorphism
$$U \triangleright \gamma(m \boxtimes n) \simeq \gamma(m \boxtimes n) \triangleleft U \ , $$
with respect to the variable $U \in \cC$. That is, the object $\gamma(m \boxtimes m) \in \cM \boxtimes_\cC \cN$ is central with respect to the bimodule structure. Any functor $\cM \underset{\max}{\boxtimes} \cN \to \cD$, $\cD$ being another $C^*$-category, is completely determined by its restriction to objects of the form $U \triangleright m \boxtimes n \triangleleft V$, for $U,V \in \cC$, and the morphisms between them. This holds in particular for the canonical balanced functor $\gamma: \cM \underset{\max}{\boxtimes} \cN \to \cM \underset{\cC}{\boxtimes} \cN$, which corresponds to the identity endofunctor of $\cM \underset{\cC}{\boxtimes} \cN$. Together with the centrality of $\gamma(m \boxtimes n)$, this implies that this element is a generating object for the relative tensor product. The following Theorem summarizes the discussion.

\begin{theorem}
Let $(\cM,m)$ and $(\cN,n)$ be as above, corresponding to Yetter-Drinfeld $C^*$-algebras $A := \underline{\End}(m)$ and $B := \underline{\End}(n)$ respectively. For the universal relative tensor product $\cM \underset{\cC}{\boxtimes} \cN$, one has canonical left and right $\cC$-actions and a canonical object $m_{A \boxtimes B}$ such that, $(\cM \underset{\cC}{\boxtimes} \cN, m_{A \boxtimes B})$ is a  cyclic $\cC$-bimodule $C^*$-category with central generator. It is defined to be the maximal braided tensor product of $(\cM,m)$ and $(\cN,n)$. 
\label{maxbraidedtensorproducttheo}
\end{theorem}

The $C^*$-algebra object in $\Vect(\cC)$ representing the braided tensor products of the module categories is denoted by $A \boxtimes B$, and it is called the braided tensor product of the Yetter-Drinfeld $C^*$-algebra objects $A$ and $B$.

\raggedright

\end{document}